\documentclass[final,10pt]{elsarticle}

\usepackage{cite}
\usepackage{amsbsy,amsfonts,amssymb,multirow}
\usepackage{color}
\usepackage{amsmath,fontenc,epsfig,subfigure,lettrine,stmaryrd,enumerate}
\usepackage[latin1]{inputenc}
\usepackage{graphics}
\usepackage{psfrag}
\usepackage[crop=pdfcrop]{pstool}
\usepackage[margin=8pt,font=small,labelfont=bf,format=plain,indention=.5cm,labelsep=quad,skip=6pt]{caption}
\usepackage{url}
\usepackage{float}
\usepackage{algorithm}
\usepackage{algorithmic}
\usepackage{fullpage}

\newcommand{\mub} {\ensuremath{\mu}}
\newcommand{\muonline} {\ensuremath{\breve\mub}}
\newcommand{\proj} {\ensuremath{P}}

\newcommand{\nt} {\ensuremath{n_t}}
\newcommand{\nin} {\ensuremath{n_i}}

\newcommand{\rk}{\ensuremath{\res^{(k)}}}

\newcommand{\tilderk}{\ensuremath{\tilde \res^{(k)}}}

\newcommand{\jk}{\ensuremath{\jac^{(k)}}}

\newcommand{\tildejk}{\ensuremath{\widetilde  {\jac^{(k)}\Phi_\state}}}

\newcommand{\pk}{\ensuremath{s^{(k)}}}

\newcommand{\oneton}{\ensuremath{\{1,\ldots,N\}}}

\newcommand{\RR}[1]{\ensuremath{\mathbb{\res}^{ #1 }}}

\newcommand{\Span}[1]{\mathrm{span}\left(#1\right)}

\newcommand{\Range}[1]{\mathrm{range}\left(#1\right)}
\newcommand{\vecmat}[2]{\left[#1^1 \ \cdots\ #1^{#2}\right]}

\newcommand{\sampleindices}{\ensuremath{\mathcal I}}

\newenvironment{proof}{\begin{trivlist}
\item {\emph {Proof}:\:}}
{$\square$\end{trivlist}}
\newenvironment{definition}[1]{\begin{trivlist}
\item {{\emph {#1}}:\:}}
{\end{trivlist}}

\newcommand{\state} {\ensuremath{w}}
\newcommand{\res} {\ensuremath{R}}
\newcommand{\jac} {\ensuremath{J}}
\newcommand{\snap} {\ensuremath{W}}

\newcommand{\params} {\ensuremath{\mub}}
\newcommand{\paramtrain} {\ensuremath{\mub^{\mathrm{train}}}}

\newcommand{\nsnap} {\ensuremath{n_\snap}}

\newcommand{\statecoords} {\ensuremath{\state_r}}
\newcommand{\nstate} {\ensuremath{n_\state}}
\newcommand{\podstate} {\ensuremath{\Phi_\state}}

\newcommand{\podres} {\ensuremath{\Phi_\res}}

\newcommand{\podjac} {\ensuremath{\Phi_\jac}}

\newcommand{\nR}[0] {\ensuremath{n_R}}
\newcommand{\nJ}[0] {\ensuremath{n_J}}

\newcommand{\nnode}[0] {\ensuremath{n_{{s}}}}
\newcommand{\nbasisgreed}[0] {\ensuremath{n_{{c}}}}
\newcommand{\nbasisgreedpit}[0] {\ensuremath{n_{{ci}}}}
\newcommand{\nbasisgreedpitmin}[0] {\ensuremath{n_{{ci,\mathrm{min}}}}}

\mathchardef\mhyphen="2D

\newcommand{\Qtot}[0] {\ensuremath{n_b}}
\newcommand{\nrhsmax}[0] {\ensuremath{n_{\mathrm{RHS}}}}
\usepackage{subfigure, algorithmic, algorithm}
\usepackage{amsmath,amssymb}
\usepackage{graphics,epsfig}
\usepackage{wrapfigure}
\usepackage{subfigmat}
\usepackage{amsopn}

\newcommand{\Rapprox}{\ensuremath{\tilde R}}	
\newcommand{\staten}{\ensuremath{\state^n}}	
\newcommand{\tn}{\ensuremath{t^n}}	
\newcommand{\statenp}{\ensuremath{\state^{n+1}}}	
\newcommand{\stateApprox}{\ensuremath{\tilde \state}}	
\newcommand{\stateApproxnp}{\ensuremath{\tilde \state^{n+1}}}	
\newcommand{\stateApproxn}{\ensuremath{\tilde \state^{n}}}	
\newcommand{\Lg}{\ensuremath{\mathcal L_{G}}}	
\newcommand{\Lgn}{\ensuremath{\mathcal L_{G}^n}}	

\mathchardef\mhyphen="2D
\newcommand{\neqpernode}{\ensuremath{\nu}}	
\newcommand{\deltaStatecoords} {\ensuremath{\state_r}}
\newcommand{\Dofs}{\ensuremath{ \delta}}	
\newcommand{\dofs}[1]{\ensuremath{ \Dofs(#1)}}	
\newcommand{\nnodeadd}{\ensuremath{n_{{a}}}}	
\newcommand{\nnodeaddpit}{\ensuremath{n_{{ai}}}}     
\newcommand{\nnodeaddpitmin}{\ensuremath{n_{{ai,\mathrm{min}}}}}     
\newcommand{\nmaskout}{\ensuremath{n_{\mathsf k}}}	
\newcommand{\nmask}{\ensuremath{n_{\mathsf j}}}	
\newcommand{\pnk}{\ensuremath{s^{n+1(k)}}}	
\newcommand{\stationarySet}{\ensuremath{\mathcal M}}	
\newcommand{\stationaryTs}[1]{\ensuremath{\mathcal M^{#1}}}	
\newcommand{\stationparams}{\ensuremath{\stationarySet^n(\params)}}	
\newcommand{\levelset}{\ensuremath{\mathcal L}}	
\newcommand{\levelsetTs}[1]{\ensuremath{\levelset^{#1}}}	
\newcommand{\lnparams}{\ensuremath{\levelset^n(\params)}}	
\newcommand{\snapshotSet}[1]{\ensuremath{\mathsf W_{#1}}}	
\newcommand{\statevariable}{\ensuremath{w}}	
\newcommand{\tnp}{\ensuremath{t^{n+1}}}	
\newcommand{\matvec}[2]{\ensuremath{\left[#1^1\ \cdots \ #1^{#2}\right]}}	
\newcommand{\nRp}{\ensuremath{\nR'}}	
\newcommand{\podresp}{\ensuremath{\podres'}}	
\newcommand{\restrictpost}{\ensuremath{\underline Z}}	
\newcommand{\restrict}[1]{\ensuremath{Z #1}}	
\newcommand{\restrictarg}[1]{\ensuremath{\mathbf Z\left(#1\right)}}	
\newcommand{\mask}[1]{\ensuremath{\bar Z #1}}	
\newcommand{\RS}{\ensuremath{\mathrm{CR}}}	
\newcommand{\WT}{\ensuremath{\mathrm{WT}}}	
\newcommand{\hatpodpseudo}[1]{\ensuremath{\left[\restrict{\Phi_{#1}}\right]^+}}	
\newcommand{\hatpodpseudop}[1]{\ensuremath{\left[\restrict{\Phi_{#1}'}\right]^+}}	
\newcommand{\hatpodrespseudo}{\ensuremath{\hatpodpseudo{R}}}	
\newcommand{\hatpodrespseudop}{\ensuremath{\hatpodpseudop{R}}}	
\newcommand{\hatpodjacpseudo}{\ensuremath{\hatpodpseudo{J}}}	
\newcommand{\iter}{\ensuremath{i}}	
\newcommand{\niter}{\ensuremath{n_\mathrm{it}}}	
\newcommand{\initialCond}{\ensuremath{\state^0}}	

\newtheorem{theorem}{Theorem}[section]

\newtheorem{proposition}[theorem]{Proposition}

\newtheorem{propos}[theorem]{Proposition}

\newenvironment{remark}[1][Remark]{\begin{trivlist}
\item[\hskip \labelsep {\bfseries #1}]}{\end{trivlist}}

\journal{Journal of Computational Physics}
\begin{document}
\begin{frontmatter}

\title{The GNAT method for nonlinear model reduction: effective implementation and application to computational fluid dynamics and turbulent flows}
\author[sandia]{Kevin Carlberg\corref{sandiacor}}
\ead{ktcarlb@sandia.gov}
\ead[url]{sandia.gov/~ktcarlb}
\address[sandia]{Sandia National Laboratories}
\cortext[sandiacor]{7011 East Ave, MS 9159, Livermore, CA 94550. Sandia is a
multiprogram laboratory operated by Sandia
Corporation, a Lockheed Martin Company, for the United States Department of
Energy under contract DE-AC04-94-AL85000.}
\author[stanford]{Charbel Farhat\corref{stanfordcor}}
\ead{cfarhat@stanford.edu}
\address[stanford]{Stanford University} \cortext[stanfordcor]{Durand Building, 496 Lomita Mall, Stanford University, Stanford, CA 94305-3035}
\author[sandia]{Julien Cortial\corref{sandiacor}}
\ead{jcortia@sandia.gov}
\author[stanford]{David Amsallem\corref{stanfordcor}}
\ead{amsallem@stanford.edu}

\begin{abstract}

The Gauss--Newton with approximated tensors (GNAT) method is a nonlinear model
reduction method that operates on fully discretized computational models. It
achieves dimension reduction by a Petrov--Galerkin projection associated with
residual minimization; it delivers computational efficency by a
hyper-reduction procedure based on the `gappy POD' technique.  Originally
presented in Ref.\ \citep{carlbergGappy}, where it was applied to implicit
nonlinear structural-dynamics models, this method is further developed here
and applied to the solution of a benchmark turbulent viscous flow problem. To
begin, this paper develops global state-space error bounds that justify the
method's design and highlight its advantages in terms of minimizing components
of these error bounds. Next, the paper introduces a `sample mesh' concept that
enables a distributed, computationally efficient implementation of the GNAT
method in finite-volume-based computational-fluid-dynamics (CFD) codes. The suitability of GNAT for
parameterized problems is highlighted with the solution of an academic problem
featuring moving discontinuities.  Finally, the capability of this method
to reduce by orders of magnitude the core-hours required for large-scale CFD
computations, while preserving accuracy, is demonstrated with the simulation
of turbulent flow over the Ahmed body.  For an instance of this benchmark
problem with over 17 million degrees of freedom, GNAT 
outperforms several other nonlinear model-reduction methods, reduces the
required computational resources by more than two orders of magnitude, and
delivers a solution that differs by less than 1\% from its high-dimensional
counterpart.

\end{abstract}


\end{frontmatter}

\section{Introduction}
\label{sec:intro}

Computational fluid dynamics (CFD) modeling and simulation tools have become indispensable in many
engineering applications due to their ability to enhance the understanding of
complex fluid systems, reduce design costs, and improve the reliability of
engineering systems. Unfortunately, many desired high-fidelity CFD simulations
are so computationally intensive that they can require unaffordable
computational resources or time to completion, even when
supercomputers with thousands of cores are available.  Consequently, such
simulations are often impractical for time-critical applications such as flow
control, design optimization, uncertainty quantification, and system
identification.

Projection-based nonlinear model-reduction methods constitute a promising
approach for bridging the gap betwen CFD and such applications.  These methods
approximate a given high-fidelity model by reducing its dimension, i.e., the
number of equations and unknowns that describe it.  For this purpose, such
methods first perform intensive large-scale computations `offline' to
construct {\it a priori} a low-dimensional subspace onto which they project
the high-dimensional model of interest. This leads to a reduced-order model
(ROM) characterized by low-dimensional operators. Then, in an `online' stage,
they exploit this ROM to compute approximate solutions that lie in this
pre-computed subspace.

Unfortunately, the computational cost associated with assembling the ROM's
low-dimensional operators --- matrices and vectors for implicit schemes, and
vectors for explicit ones --- scales with the large dimension of the underlying
high-dimensional model. For this reason, projection-based
model-reduction methods are efficient primarily for problems where the
aforementioned operators must be constructed only once, or can be assembled
\emph{a priori}. These include linear time-invariant 
systems \citep{antoulas2005approximation,amsallem2009method}, linear
stationary systems whose operators are affine functions of the input
parameters \citep{prud2002reliable,rozza2007reduced}, and a class of nonlinear
systems characterized by quadratic nonlinearities
\citep{veroy2003peb,nguyen2005certified,veroy2005certified}.  Within these
contexts, projection-based model-reduction methods have been successfully
applied to problems in structural dynamics \citep{amsallem2009method},
aerodynamics\ \citep{HTD99,LA,HTD00,WP,E03} and aeroelasticity\
\citep{TDH,KHBS,LFL,LF,Amsallem:2008bi,Amsallem:2010ci}, among others.

On the other hand, when projection is applied to linear time-varying systems,
linear stationary systems with nonaffine parameter dependence, or general
nonlinear problems, the resulting ROM is costly to assemble.  This high cost
arises from the need to evaluate the high-dimensional nonlinear function (and
possibly its Jacobian) at each computational step of a solution algorithm. To
overcome this roadblock, several approaches have been proposed. Such
complexity-reduction techniques are sometimes referred to as hyper-reduction
methods \citep{ryckelynck2005phm,NME:NME4371}.  Several of these techniques
are outlined below.

The `empirical interpolation' method developed for linear elliptic problems
with non-affine parameter dependence \citep{barrault2004eim}, as well as for
nonlinear elliptic and parabolic problems \citep{grepl2007erb}, reduces the
computational cost associated with nonlinearities by interpolating the
governing nonlinear function at a few spatial locations using an empirically
derived basis. This method operates directly on the governing partial
differential equation (PDE) and therefore at the continuous level. Its variant
proposed in Ref.~\citep{PeraireRBnonlinear} relies for the same purpose on
`best [interpolation] points' and a POD basis.  The hyper-reduction method
proposed in Ref.~\citep{haasdonkExplicit,haasdonk2008reduced} extends the
empirical-interpolation method to the case of a scalar conservation law
discretized by an explicit solution algorithm.  However, no further extension
of this method to CFD has been performed, as this is not a straightforward
task. 

Alternatively, the trajectory piece-wise linear (TPWL) method developed in
\citep{tpwl} operates at the semi-discrete level, i.e., on the ordinary
differential equation (ODE) obtained after discretizing the PDE in space. TPWL
constructs a ROM as a weighted combination of linearized models, where each
linearization point lies on a training trajectory.  However, because the
resulting ROM never queries the underlying high-dimensional model away from
the linearization points, this method in principle lacks robustness for highly
nonlinear problems such as those arising from a large class of CFD
applications.

Another class of hyper-reduction methods reduces computational complexity by
computing only a few entries of the nonlinear function (and possibly its
Jacobian) appearing in the governing ODE. In this paper, such methods are
referred to as function-sampling methods; these methods also operate at the
semi-discrete level. They include collocation approaches,
which compute a small subset of the entries of the residual vector. Ref.~\citep{LeGresleyThesis} proposes collocation of the nonlinear equations
followed by a least-squares solution of the resulting overdetermined system of
nonlinear equations. Similarly, Refs.~\citep{astrid2007mpe,ryckelynck2005phm}
propose a collocation of the equations followed by a Galerkin projection.
Function-reconstruction approaches define another subset of function-sampling
methods. In contrast to collocation methods, these techniques use the sampled
entries of the nonlinear function to approximate the entire nonlinear function
by interpolation or least-squares regression.  One such method reconstructs
the governing nonlinear function in the least-squares sense, using the same
basis adopted to represent the state. This approach was developed in for
parameter-varying systems \citep{astrid2007mpe}\footnote{For such problems,
this method is mathematically equivalent to collocation followed by a Galerkin
projection when the resulting overdetermined system is solved by a Galerkin
projection method.} and for nonlinear dynamical systems discretized by
explicit time-integration schemes \citep{bos2004als}. Other
function-reconstruction approaches include the semi-discrete analogs to the
empirical and best points interpolation methods that have been developed for
parameterized nonlinear stationary problems
\citep{chaturantabut2010journal,galbally2009non}, and for nonlinear dynamics
problems \citep{chaturantabut2010journal,drohmannEOI}.

Although some of the function-reconstruction methods outlined above have been
successfully applied to large-scale steady-state problems (e.g., see Ref.~\citep{galbally2009non}), it will be shown that many of them lack the
robustness needed to address highly nonlinear dynamical systems such as those
arising from unsteady CFD applications.

The Gauss--Newton with approximated tensors (GNAT) method
\citep{carlbergGappy} is a nonlinear Petrov--Galerkin projection method
equipped with a function-sampling hyper-reduction scheme.  In constrast with
the aforementioned nonlinear model-reduction methods, it operates at the
discrete level, i.e., on the system of nonlinear equations arising at each
time step, which is obtained after discretizing the PDE in both space and
time.  GNAT is designed around approximations that satisfy consistency and
discrete-optimality conditions.  As such, it has been successfully applied to
nonlinear structural-dynamics problems \citep{carlbergGappy} for which it
demonstrated robustness, accuracy, and excellent CPU performance. 

This work further develops the GNAT methodology and demonstrates its potential
for CFD applications. Section~\ref{sec:PF} formulates the problem of interest,
	and Section~\ref{sec:GO} provides an overview of GNAT.  Within this
	overview, a new consistent snapshot-collection procedure is proposed in
	Section~\ref{sec:consistentProj}. Next, Section~\ref{sec:errBackEuler}
	presents global error bounds for the discrete fluid state in the case of the
	implicit backward-Euler scheme.  This time integrator may be of little
	practical importance to CFD, but the developed error bounds highlight the
	merits of the principles underlying the construction of a GNAT ROM. Then,
	Section~\ref{sec:implementation} proposes a simple yet effective
	implementation of GNAT's online stage on parallel computing platforms.  This
	implementation features the concept of a `sample mesh', which is a carefully
	chosen {\it tiny} subset of the original CFD mesh on which all online GNAT
	computations are performed.  The sample mesh has few connectivity
	requirements and is therefore easily amenable to partitioning for parallel
	distributed computations. Most importantly, the implementation can be
	tailored to any specific CFD scheme or software, and to the fast computation
	of outputs such as pressure coefficients, lift, and drag.
	Section~\ref{sec:experiments} demonstrates the potential of GNAT to
	effectively reduce the dimension and complexity of highly nonlinear CFD
	models while maintaining a high level of accuracy. In
	Section~\ref{sec:burgers}, GNAT is applied to a parameterized hyperbolic
	problem featuring a moving shock; GNAT's online performance is tested for
	parameter values different from those used to collect simulation data
	offline. Section~\ref{sec:ahmedFine} demonstrates the potential of GNAT for
	challenging CFD applications with the fast solution of a benchmark turbulent
	flow problem with over 17 million unknowns.  For this problem, GNAT
	outperforms many of the aforementioned nonlinear model-reduction methods.
	It reproduces the solution delivered by the high-dimensional CFD model with
	less than 1\% discrepancy, while reducing the associated computational cost
	in core-hours by more than two orders of magnitude.  Finally,
	Section~\ref{sec:Conc} offers conclusions.
\section{Problem formulation}
\label{sec:PF}

\subsection{Parameterized nonlinear CFD problem}

Consider the ODE resulting from semi-discretizing the conservation
form of the compressible Navier-Stokes equations by a finite difference,
finite volume, or stabilized finite element method  --- possibly augmented by
a turbulence model --- and a given set of boundary conditions:
\begin{align}\label{eq:ODE}
\begin{split}
\frac{d\state}{dt}& = \displaystyle{F\left(\state(t),t;\params\right)}\\
\state(0)& = \state^0(\params).
\end{split}
\end{align}
Let
\begin{align} \label{eq:outODE}
\begin{split}
z &= H\left(\state(t),\params\right)\\
& = L(\params)
\end{split}
\end{align}
denote the outputs of interest that may include the lift, drag, and other
quantitites.

Here, $t\in\RR{+}$ denotes time, $\state \in\RR{N}$ denotes the discrete fluid
state vector (i.e., the vector of discrete conserved fluid variables), $N$
designates the dimension of the semi-discretization and is typically large,
$\state^0:\RR{d} \rightarrow\RR{N}$ denotes the parameterized initial
condition, and $F:\RR{N}\times\RR{+}\times \RR{d}\rightarrow \RR{N}$ is the
nonlinear function arising from the semi-discretization of the convective and
diffusive fluxes and source term (when present). The $d$ input parameters,
which may include shape parameters, free-stream conditions, and other design
or analysis parameters of interest, are denoted by $\params \in \RR{d}$. The
(feasible) input-parameter domain is denoted by $\mathcal D$, and therefore
$\params\in \mathcal D\subset \RR{d}$. $H:\RR{N}\times \RR{d}\rightarrow
\RR{p}$ and $L:(\params)\mapsto H(\state(t;\params),\params)$ define two
equivalent mappings for the output vector $z\in\RR{p}$.

In this work, attention is occasionally focused on a finite-volume
semi-discretization method operating on a dual CFD mesh. Consequently, the
sampling concepts discussed in Sections~\ref{sec:GO} and
\ref{sec:implementation} are node/cell oriented. However, all concepts,
algorithms, and techniques presented in this paper are easily extendible to
finite-difference and stabilized finite-element semi-discretization methods.

\subsection{Objective: time-critical analysis}\label{sec:realtime}

Consider the following objective: given inputs $\muonline\in\mathcal D$,
compute fast approximations of the outputs $\tilde L(\muonline) \approx
L(\muonline)$, where ``fast'' is defined in one of the following senses:
 \begin{enumerate} 
 \item The evaluation takes a sufficiently small amount of \emph{time}. This
 is relevant to applications that demand near-real-time analysis, where the
 objective is to compute outputs in a time below a threshold value; the number
 of computational cores required to perform the analysis is not a primary
 concern. Examples include flow control and routine analysis, where the
 analyst may require an answer within a given time frame.
\item The evaluation consumes a sufficiently small amount of
\emph{computational resources}, as measured by computational cores multiplied
by time. This is relevant to many-query applications, where the objective is
to evaluate the model at as many points in the input space as possible, given
a fixed amount of wall time and processors. Examples include aerodynamic shape
optimization and uncertainty quantification.
 \end{enumerate}
When the number of degrees of freedom $N$ of the high-dimensional CFD model is
sufficiently large, solving the state equations (\ref{eq:ODE}) with
$\params = \muonline$ and subsequently computing the outputs by
Eq.~\eqref{eq:outODE} becomes
prohibitively time- and resource-intensive for time-critical applications.

Instead, the following two-stage offline--online strategy can be employed.
During the offline stage, Eq.~(\ref{eq:ODE}) is solved for
$d_{\mathrm{train}}$ points in the input-parameter domain; this defines the
training domain $\mathcal D_{\mathrm{train}} \equiv \{\params^j\}_{j =
1}^{d_{\mathrm{train}}}\subset \mathcal D$. The data acquired during these
training simulations are then used to construct a surrogate model that is
capable of rapidly reproducing the behavior of the high-dimensional model at
arbitrary points in $\mathcal D$.  The online stage uses this surrogate model
to perform time-critical analysis for inputs $\muonline$ with $\muonline\not\in\mathcal
D_{\mathrm{train}}$ in general.

In contrast to surrogate-modeling techniques based on data-fit approximations
(e.g., response surfaces) of the input--output map $L(\params)$,
model-reduction methods take a more physics-based approach. These techniques
aim to achieve time-critical analysis by \emph{approximately} solving the
(physics-based) state equations for the online inputs $\muonline$ and then
computing the resulting outputs. The next section provides an overview of the
GNAT model-reduction method developed for this purpose.
\section{Overview of the GNAT model-reduction method}
\label{sec:GO}
To keep this paper as self-contained as possible, this section
provides an overview of the Gauss--Newton with approximated tensors (GNAT)
nonlinear model-reduction method first proposed in Ref.~\citep{carlbergGappy}.

\subsection{Computational strategy and numerical properties}
\label{sec:consistency}

Nonlinear model reduction is often performed in a somewhat \emph{ad hoc}
manner.  As a result, nonlinear ROMs often lack important numerical
properties. To avoid this pitfall, the design of the GNAT method employs a
computational strategy that constructs approximations to meet conditions
related to notions of \emph{discrete optimality} and \emph{consistency}.

The design of GNAT is based on the premise that if a given computational model
is unaffordable for a given time-critical application, approximations can be
introduced to this model to make it more economical yet retain
an appropriate level of accuracy.
This premise leads to a hierarchy of models characterized by
tradeoffs between accuracy and computational complexity.  Then, the objective
is to construct computationally efficient approximations that introduce
minimal error with respect to the previous model in the hierarchy. GNAT
achieves this objective by relying on the concepts of discrete optimality and
consistency introduced in Ref.~\citep{carlbergGappy}. Both concepts are
recalled below.

\begin{definition}{Discrete-optimal approximation}
Here, an approximation is said to be discrete optimal if it leads to
approximations that --- at the discrete level --- minimize an error measure
associated with the previous model in the hierarchy. This ensures that some
measure of the error in the discrete approximation decreases monotonically as
the approximation spaces expand (a property also referred to as \emph{a
priori} convergence).
\end{definition}

\begin{definition}{Consistent approximation}
Here, an approximation is said to be consistent if, when implemented without
snapshot compression, it introduces no additional error in the solution at the
training inputs.
\end{definition}

Figure \ref{fig:strategy} depicts the model hierarchy employed by GNAT. This
hierarchy consists of three computational models: the original tier
$\mathrm{I}$ high-dimensional model and two increasingly approximated models
referred to as the tier $\mathrm{II}$ and tier $\mathrm{III}$ (reduced-order)
models, respectively.  Each of these reduced-order models is generated by 1)
acquiring snapshots from simulations performed for training inputs $\mathcal
D_{\mathrm{train}}$ using the previous (i.e., more accurate) model in the
hierarchy, 2) compressing the snapshots, and 3) introducing an approximation
that exploits the compressed snapshots.

\begin{figure}[htbp]
\centering
\includegraphics[width=11cm]{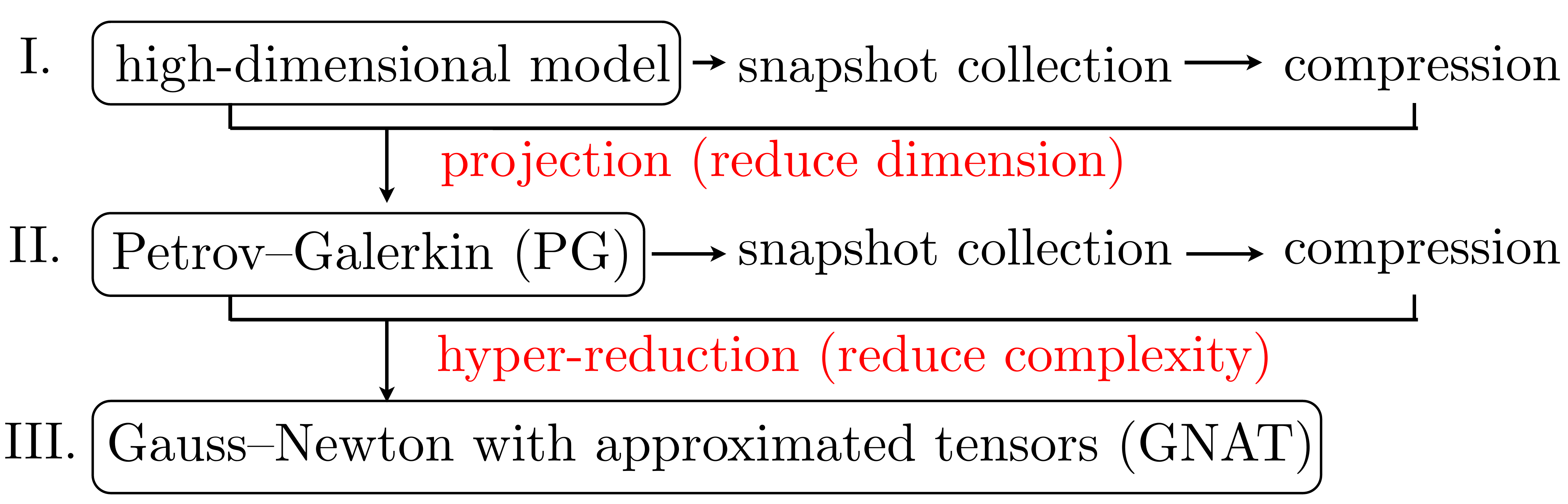}
\caption{Model hierarchy with approximations shown in red.}
\label{fig:strategy}
\end{figure}

\subsection{Fully discrete computational framework}\label{sec:fullyDiscrete}

As stated in the introduction, GNAT operates at the discrete level. That is,
the method introduces approximations after ODE \eqref{eq:ODE} has been
discretized in time. Hence, the ROM constructed by GNAT is a low-dimensional
algebraic system that governs the solution of ODE~\eqref{eq:ODE} at each time
step; it cannot generally be expressed as a low-dimensional ODE.  As such,
GNAT may be less convenient than other model reduction methods: the ROM it
produces is valid only for the time-integrator adopted for the
high-dimensional model.  However, Section \ref{sec:projection} will show that
the fully discrete framework enables GNAT to achieve discrete optimality.

Throughout the remainder of this paper, it is assumed that Eq.\ \eqref{eq:ODE}
is solved by an implicit linear multi-step time integrator. In this case, if
$\nt$ time steps are carried out, a sequence of $\nt$ systems of nonlinear
equations arises.  Each of these systems can be written as
\begin{equation}\label{eq:FOMdyn}
R^n(\state^{n+1};\params) = 0
\end{equation}
for $n=0,\ldots,\nt-1$, with outputs 
\begin{align} \label{eq:outdyn}
\begin{split}
z & = G(\state^0,\ldots, \state^{\nt},\params).
\end{split}
\end{align}
Here, a superscript $n$ designates the value of a variable at time step $n$,
the operators $R^n:\RR{N}\times \RR{d}\rightarrow \RR{N}$ for $n=0,\ldots,
\nt-1$ are nonlinear in both arguments, and $G:\RR{N}\times\cdots
\times\RR{N}\times \RR{d}\rightarrow \RR{p}$. The fluid state vectors
$\state^n$, $n=1,\ldots,\nt$ are implicitly defined by Eq.\  (\ref{eq:FOMdyn})
for a given $\params$, and $\state^0=\state^0(\params)$ is given by the
	initial condition. 

For simplicity, consider one instance of Eq.\  (\ref{eq:FOMdyn}) defined by
one time instance and one vector of input parameters. Such an instance can be
written as
\begin{equation}
\label{eq:FOM} R(\state) = 0.
\end{equation}
Here, the fluid state vector $\state\in\RR{N}$ is implicitly defined by
Eq.~(\ref{eq:FOM}), and $R:\RR{N}\rightarrow \RR{N}$ with $\statevariable
\mapsto R(\statevariable)$ is a nonlinear mapping. In the remainder of this
paper, Eq.~(\ref{eq:FOM}) is associated with the high-dimensional CFD model
(tier $\mathrm{I}$ in Figure \ref{fig:strategy}).

\subsection{Petrov--Galerkin projection}\label{sec:projection}

To reduce the dimension of Eq.~(\ref{eq:FOM}), the GNAT method employs a
projection process. This leads to the tier $\mathrm{II}$ ROM in the model
hierarchy. Specifically, GNAT seeks an approximate solution $\tilde \state$ to
Eq.~(\ref{eq:FOM}) in the affine trial subspace $\initialCond + \mathcal
W\subset \RR{N}$ of dimension $n_\state\ll N$. Hence, $\tilde \state$ can be
written as
\begin{equation}\label{eq:yApprox}
\tilde \state = \initialCond + \podstate \state_r,
\end{equation}
where $\podstate\in\RR{N\times n_{\state}}$ is a matrix
representing an $\nstate$-dimensional basis for $\mathcal W$, and
$\state_r\in\RR{n_\state}$ denotes the generalized coordinates of the fluid
state vector in this basis.
Note that the \emph{increment} in the state
$\tilde \state - \initialCond$ and not the state itself is sought in the
subspace $\mathcal W$.  This is an important consideration when defining the
basis $\podstate$ as will be described in Section \ref{sec:consistentProj}.

\subsubsection{Discrete optimality}\label{sec:leftproj}
Substituting Eq.\ \eqref{eq:yApprox} into Eq.\ \eqref{eq:FOM} yields
$R(\initialCond + \podstate\state_r)=0$, which represents an overdetermined
system of $N$ equations in $\nstate$ unknowns. Consequently, the GNAT method
computes $\tilde \state$ as the solution to the minimization problem
\begin{gather}\label{eq:GNPG}
\underset{\bar \state\in \initialCond+\mathcal W}{\mathrm{minimize}}\
\|R(\bar \state)\|_2.
\end{gather}
GNAT solves this nonlinear least-squares problem by the Gauss--Newton
method, which is globally convergent under certain assumptions. This method
delivers a solution that is \emph{discrete optimal at each time step}: the
solution minimizes the discrete residual associated with the tier $\mathrm{I}$
model over the trial subspace.  This residual-minimization approach is
mathematically equivalent to performing a Petrov--Galerkin projection with a
test basis corresponding to $\displaystyle\frac{\partial R}{\partial
\state}\podstate$ (see Ref.~\citep{carlbergGappy}). Note that the
fully discrete framework described in Section \ref{sec:fullyDiscrete} enables
discrete optimality to be achieved: the test basis depends on the discrete
residual and thus depends on the time integrator, so it is not defined at the
semi-discrete level.

Ref.~\citep{carlbergGappy} numerically demonstrated the superior accuracy
delivered by the tier II Petrov--Galerkin ROM associated with Eq.\ \eqref{eq:GNPG} compared with its
counterpart based on a (more commonly used) Galerkin projection when applied
to a non-self-adjoint problem characterized by an unsymmetric residual
Jacobian, which is typical in CFD.  This strong performance is likely due to
the discrete optimality property, which Galerkin projection lacks for such
problems. However, for the relatively small class of CFD problems
characterized by a symmetric residual Jacobian, a Galerkin projection is also
discrete optimal, as it minimizes the discrete residual, albeit for a different norm
\citep{carlbergGappy}.

\subsubsection{Consistency}\label{sec:consistentProj}
To ensure a consistent projection, the basis $\podstate$ can be computed by
proper orthogonal decomposition (POD) using a specific set of snapshots
collected from simulations performed using the tier $\mathrm{I}$
CFD model at the training inputs.  Specifically, given a
snapshot matrix $W \in \RR{N\times \nsnap}$, a POD basis $\Phi\in\RR{N\times
n_{\Phi}}$ of dimension $n_{\Phi}\leq \nsnap$ is obtained by first computing
the (thin) singular value decomposition (SVD)
 \begin{equation} 
	 W = U\Sigma V^T,
  \end{equation} 
	where the superscript $T$ designates the transpose, the left-singular-vector matrix $U\equiv\matvec{u}{\nsnap}\in \RR{N\times
	\nsnap}$ satisfies $U^TU = I$, the singular-value matrix $\Sigma =
	\mathrm{diag}(\sigma_i)$ satisfies
	$\sigma_1\geq\sigma_2\geq\cdots\geq \sigma_{\nsnap}\geq 0$, and the right-singular-vector matrix $V\in\RR{\nsnap\times
	\nsnap} $ satisfies $V^TV=I$. Then, the sought-after POD basis is obtained by selecting the first
	$n_{\Phi}\leq \nsnap$ left singular vectors: $\Phi = \matvec{u}{n_{\Phi}}$. As a result,
	$\Phi$ has orthonormal columns and satisfies $\Phi^T\Phi = I$. Often, $n_{\Phi}$ is
	determined from an energy criterion such that the POD basis captures a
	fraction of the statistical energy of the snapshots.

Ref.~\citep{carlbergGappy} proved that the Petrov--Galerkin projection defined
by Eq.~\eqref{eq:GNPG} is consistent when $\podstate$ is computed by POD with
snapshots of the form $\{\state^n(\params)-\state^{n(0)}(\params) \ | \
n=1,\ldots,\nt, \params \in \mathcal D_{\mathrm{train}}\}$, where
$\state^{n(0)} = \state^{n-1}$, $n=1,\ldots,\nt$  denotes the initial guess for the Newton
solver. This implies that the snapshots used for POD should correspond to the
solution \emph{increment} at each time step of the training simulations. This
remark is noteworthy because most nonlinear model-reduction techniques reported in
the literature employ a POD basis computed using snapshots
$\{\state^n(\params) \ | \ n=0,\ldots,\nt, \params \in \mathcal
D_{\mathrm{train}}\}$, which do not lead to a consistent projection. 

Here, an alternative snapshot-collection procedure is proposed:
$\{\state^n(\params)-\state^{0}(\params)\ |\ n=1,\ldots,\nt, \params \in
\mathcal D_{\mathrm{train}}\}$. These snapshots also lead to a consistent
projection under certain conditions. \ref{app:snapshotConsistency} discusses
these conditions and contains the proof that both snapshot-collection
procedures lead to a consistent Petrov--Galerkin projection.

\subsection{Hyper-reduction}\label{sec:systemApprox}

Solving the least-squares problem (\ref{eq:GNPG}) by the Gauss--Newton method
leads to the following iterations: for $k=1,\ldots,K$, solve the linear
least-squares problem
\begin{equation}\label{eq:newtonRed1PGLS} \pk = \arg\min_{a\in\RR{n_\state}}
\|\jk\podstate a + \rk\|_2 
\end{equation} 
and set
\begin{equation}
\label{eq:newtonRed2PGLS}\state_r^{(k+1)} =
\state_r^{(k)}+\alpha^{(k)}\pk, \end{equation} 
where $K$ is determined by the
satisfaction of a convergence criterion, $\state_r^{(0)}$ is the initial
guess (often taken to be the generalized coordinates computed at the previous time step), and $\rk\equiv
R\left(\initialCond + \podstate\statecoords^{(k)}\right)$ and $\displaystyle\jk\equiv
\frac{\partial R}{\partial w}\left(\initialCond +
\podstate\statecoords^{(k)}\right)$ are the nonlinear residual and its
Jacobian at iteration $k$, respectively. The step length $\alpha^{(k)}$
is computed by executing a line search in the direction $\pk$ to ensure
convergence, or is set to the canonical step length of unity. Even though the
dimension of the trial subspace is small, the computational cost of solving
the above nonlinear least-squares problem scales with the dimension $N$ of the
tier $\mathrm{I}$ high-dimensional CFD model. As mentioned in the introduction, this is the computational
bottleneck faced by many (if not all) projection-based nonlinear model
reduction techniques. The role of hyper-reduction (referred to as system
approximation in Ref.~\citep{carlbergGappy}) is to decrease this computational
cost.

\subsubsection{Optimality}\label{sec:sysApproxOptimality}
To address the performance bottleneck identified above, GNAT employs the gappy
POD data reconstruction technique~\citep{sirovichOrigGappy}. In the context of
GNAT, gappy POD leads to approximations of the one- and two-dimensional
tensors $\rk$ and $\jk\podstate$, respectively, by computing only a small
subset of their rows. Denoting by $\sampleindices\equiv\{\mathsf i_1, \mathsf
i_2,\ldots, \mathsf i_{\nin}\}\subset\oneton$ the set of $\nin$ sample indices
for which these functions are evaluated, the sample matrix is defined as 
 \begin{equation} \label{eq:Z}
\restrict{} \equiv \restrictarg{\mathcal I}\in\RR{N\times\nin},
 \end{equation} 
where
\begin{align}
\restrictarg{\mathcal Y}&\equiv \left[e_{\mathsf{y}_1}\ \cdots\ e_{\mathsf{y}_{n_{\mathsf y}}}\right]^T\nonumber\\
\mathcal Y&\equiv\{{\mathsf{y}_1}, \ldots\, {\mathsf{y}_{n_{\mathsf y}}}\},
\end{align}
and $e_i$ is the $i$th canonical unit vector.

Given these sample indices and bases $\Phi_R\in \RR{N\times \nR}$ and
$\Phi_J\in\RR{N\times\nJ}$, GNAT approximates $\rk$ and $\jk\podstate$ via
gappy POD as follows:
\begin{align}
\label{eq:sa1}\tilderk &= \Phi_R\hatpodrespseudo \restrict{\rk} \\
\label{eq:sa2}\tildejk &= \Phi_J\hatpodjacpseudo \restrict{\jk\podstate},
\end{align}
where the superscript $+$ designates the Moore--Penrose pseudo-inverse.
Approximations \eqref{eq:sa1}--\eqref{eq:sa2} are discrete optimal in the sense that the error
measures $\|\restrict{\rk}-\restrict{\tilderk}\|_2$ and
$\|\restrict{\jk\podstate}-\restrict{\tildejk}\|_F$
monotonically decrease as columns are added to $\podres$ and
$\podjac$, respectively.

Substituting $\rk=\tilderk$ and $\jk\podstate = \tildejk$ in Eqs.\
(\ref{eq:newtonRed1PGLS})--(\ref{eq:newtonRed2PGLS}) and assuming that
$\podjac^T\podjac = I_{n_J}$ --- which is easily achievable by computing
$\podjac$ using POD --- the tier $\mathrm{II}$ 
Petrov--Galerkin iterations are transformed into the tier $\mathrm{III}$ GNAT
iterations
\begin{align}\label{eq:newtonRed1approx}
\pk &= \arg\min_{\upsilon \in\RR{n_\state}} \|A\restrict{\jk\podstate} \upsilon +
B\restrict{\rk}\|_2\\
\label{eq:newtonRed2approx}\state_r^{(k+1)} &= \state_r^{(k)}+\alpha^{(k)}\pk.
\end{align} 
Here, the matrices $A = \hatpodjacpseudo\in \RR{n_J\times \nin}$ and $B =
\Phi_J^T\Phi_R\hatpodrespseudo\in\RR{n_J\times\nin}$ can be computed \textit{a
priori}, during the offline stage.

Note that in CFD, the Jacobian matrix is usually sparse. Hence, computing
$\restrict{\rk}$ and $\restrict{\jk\podstate}$ does not require access to all
entries of the fluid state vector. For this reason, let $\mathcal J$ denote the
minimum-cardinality set of indices of the fluid state vector that influences
the residual entries corresponding to sample-index set $\sampleindices$. 
GNAT requires access to only the `masked' state $\mask{^T}\mask{\state}$, where
$\mask\equiv \restrictarg{\mathcal J}$ and $\nmask\equiv |\mathcal
J|$.  The algebraic products implied by this masked state vector can be obtained
by evaluating only the $\mathcal J$ entries of the state $\state$. 

After the index sets $\mathcal I$ and $\mathcal J$ are determined (see
Section~\ref{sec:sampleNodeSelection} for a discussion on determining these
index sets), the online GNAT iterations executed at each time step can proceed
as follows:
\begin{enumerate}
	\item Compute $\mask \tilde \state^{(k)} = \mask\initialCond+ \mask\podstate \statecoords^{(k)}$, which
	requires updating only $\nmask$ entries of the state.
	\item Compute $\displaystyle{C^{(k)}=\restrict\frac{\partial R}{\partial
	w}\left(\mask^T\mask \tilde \state^{(k)}\right)\mask^T\mask\podstate}$ and
	$D^{(k)}=\restrict{R(\mask^T\mask
	\tilde \state^{(k)})}$,
	which necessitates computing only $\nin$ rows of $\rk$ and $\jk\podstate$,
	respectively.
	\item Compute the low-dimensional products $AC^{(k)}$ and $BD^{(k)}$.
	\item Solve the reduced-order least-squares problem
$ \pk = \arg\min\limits_{\upsilon \in\RR{n_\state}} \|AC^{(k)} \upsilon + BD^{(k)}\|_2$.
	\item Update the $\nstate$ generalized coordinates $\state_r^{(k+1)}$ using Eq.\ \eqref{eq:newtonRed2approx}.
\end{enumerate}
Because none of the above computations scales with the large dimension $N$ of
the high-dimensional CFD model, the cost of the online stage of GNAT is
typically very small.

\subsubsection{Consistency}\label{sec:gapConsistency}
To ensure consistency in the hyper-reduction procedure outlined above, the bases
$\Phi_R$ and $\Phi_J$ can be constructed using POD with snapshots that verify
certain properties. For example, Ref.~\citep{carlbergGappy} introduced three
conditions on the snapshots that together ensure consistency. These
conditions lead to a hierarchy of snapshot-collection procedures that trade
consistency for more affordable offline resources such as core-hours and
storage. Table \ref{table:snapMethods} summarizes these procedures.

\begin{table}[htd]
\centering
\begin{tabular}{||c||c|c|c|c||}
\hline
Procedure identifier & 0 & 1 & 2  & 3 \\
\hline
Snapshots for $\rk$ &  $\rk_{\mathrm I}$& $\rk_{\mathrm{II}}$ &
$\rk_{\mathrm{II}}$& $\rk_{\mathrm{II}}$\\
Snapshots for $\jk\podstate$ & $\rk_{\mathrm I}$&$\rk_{\mathrm{II}}$&
$\left[\jk\podstate\pk\right]_{\mathrm{II}}$ &
$\left[\jk\podstate\right]_{\mathrm{II}}$\\
\hline
\# of simulations per training input& 1 &2&2 & 2\\
\# of snapshots per Newton iteration & 1 & 1 & 2 & $n_\state+1$\\
\# of conditions for consistency satisfied & 0 &1&2 & 3\\
\hline
\end{tabular}
\caption{Snapshot-collection procedures for the tier $\mathrm{III}$ GNAT ROM.
The indicated snapshots are saved at each Newton iteration for the training simulations. Subscripts $\mathrm{I}$ and 
$\mathrm{II}$ specify the tier of the model for which the snapshots are collected.}
\label{table:snapMethods}
\end{table}

Procedure 3 ensures a consistent hyper-reduction because it satisfies all three
conditions that together are sufficient for consistency.  However, it is
infeasible for most problems as it requires storing either $n_\state+1$
vectors or the (sparse) high-dimensional residual Jacobian at each
Newton step of the training simulations. 

Procedure 2 does not provide a consistent hyper-reduction, although it satisfies two of the three
consistency conditions.  Furthermore, it is computationally feasible as it
requires saving only two vectors per Newton step.

Procedure 1 is more economical than procedure 2, as it requires saving only one
vector per Newton iteration and it computes one fewer POD basis. Further,
procedure 1 uses the same POD basis for the residual and its Jacobian; when this
occurs, the GNAT iterations are equivalent to the Gauss--Newton iterations for
minimizing $\tilderk$, which can abet convergence (see \ref{sec:SCP}).
However, procedure 1 satisfies only one of the three consistency conditions.

Procedure 0 requires performing only one tier $\mathrm{I}$ simulation for each
training input and therefore is similar to conventional approaches for
collecting snapshots~\citep{barrault2004eim,grepl2007erb,PeraireRBnonlinear,
chaturantabut2010journal,galbally2009non}; however, it satisfies none of  the
aforementioned consistency conditions.  \ref{sec:SCP} offers an
additional discussion of these snapshot-collection procedures.

\subsection{Computation of outputs}\label{sec:outputComputation}

After GNAT computes generalized coordinates $\statecoords^n$,
$n=1,\ldots,\nt$ for online inputs $\muonline\in\mathcal D$, the outputs $z$
can be computed. Because the objective is to ensure that no online computation
scales with the large dimension $N$ of the high-dimensional CFD model, an
alternative method to computing the outputs via Eq.\
\eqref{eq:outputExpensive} below is needed:
\begin{equation} \label{eq:outputExpensive}
	 z = G(\initialCond,\initialCond
	 +\podstate\statecoords^1,\ldots,\initialCond
	 +\podstate\statecoords^{\nt},\muonline).
  \end{equation} 
Indeed, this expression implies matrix--vector products of the form
$\podstate y$ that entail $\mathcal O(N\nstate)$ operations. 

In general, the outputs cannot be computed during the online stage of a GNAT
simulation, because the outputs may depend on entries of the state vector that are
not included in $\mathcal J$. For example, the drag force exerted on an
immersed body depends on the conserved fluid variables at all nodes located on
its wet surface, but the index set $\mathcal J$ will not generally include all
of these indices.

Instead, the outputs can be efficiently computed in a post-processing step
that accesses only the computed generalized coordinates 
$\statecoords^n$ for $n=1,\ldots,\nt$, and the rows of the initial condition
$\initialCond$ and POD basis $\podstate$ needed for the desired output
computation. To this effect, let $\mathcal K$ denote the minimum-cardinality
set of indices of the fluid state vector that affects the output computation.
Given generalized coordinates computed by GNAT online, the outputs can be
computed as
 \begin{equation} 
	 z =
	 G(\restrictpost^T\restrictpost\initialCond,\restrictpost^T\restrictpost\initialCond
	 +\restrictpost^T\restrictpost\podstate\statecoords^1,\ldots,\restrictpost^T\restrictpost\initialCond
	 +\restrictpost^T\restrictpost\podstate\statecoords^{\nt},\muonline),
  \end{equation} 
where
	$\restrictpost\equiv\restrictarg{\mathcal K}$ and
	$\nmaskout\equiv|\mathcal K|$. This approach entails products of
	the form $\restrictpost^T\restrictpost\podstate y$ that require performing
	computations with only the $\mathcal K$ rows of $\podstate$ and incur
	$\mathcal O(\nmaskout \nstate)$ operations. This operation count is small if $\nmaskout \ll N$. Fortunately, this condition holds in the case of
        spatially local outputs such as the value of flow variables at several points in the domain, or the lift and drag, which are associated with the wet surface 
        of a body. This condition does not hold for spatially global outputs.

\subsection{Offline--online decomposition}\label{sec:offlineOnline}

In summary, GNAT builds a `global' ROM --- that is, a ROM trained at multiple
points in the input-parameter space --- and performs a ROM simulation in two stages as follows:\\

\emph{Offline stage}
\begin{enumerate}
\item Perform tier $\mathrm{I}$ simulations at various training inputs
$\mathcal D_{\mathrm{train}}$.  Collect snapshots of the fluid state vector
(and snapshots of the residual if using snapshot-collection procedure 0 of
Table \ref{table:snapMethods}).
during these training simulations.
\item Compute POD basis $\podstate$ using the collected fluid-state
snapshots. 
\item If snapshot-collection procedure 1, 2, or 3 is employed, perform tier
$\mathrm{II}$  simulations at training inputs $\mathcal D_{\mathrm{train}}$.
Collect snapshots for the residual and its Jacobian during these training
simulations as specified by Table \ref{table:snapMethods} .
\item Compute POD bases $\podres$ and $\podjac$ using the
collected snapshots. To ensure the matrix $A\restrict{\jk\podstate}$ has full
rank, enforce $n_J\ge n_\state$.
\item\label{step:sampleIndices} Determine the sample-index set $\mathcal I$
(for example, see Section \ref{sec:sampleNodeSelection}), and consequently 
index set $\mathcal J$. To ensure uniqueness for the
gappy POD approximations, enforce $\nin\ge \nR$ and $\nin\ge n_J$.
\item\label{step:offlinematrices} Compute matrices $A = \hatpodjacpseudo$
and $B = \Phi_J^T\Phi_R\hatpodrespseudo$ to be used during online
computations.
\item\label{step:sampleMesh} Determine the index set $\mathcal K$ related to
output computation. 
\item\label{step:sampleMesh} Using index sets $\mathcal I$, $\mathcal J$, and $\mathcal K$,
construct the associated sample meshes (see Section~\ref{sec:implementation}) ---
which are tiny subsets of the CFD mesh --- on which to perform the online
stage of a GNAT simulation as summarized below.
\end{enumerate}
\emph{Online stage}
 \begin{enumerate} 
 \item\label{step:onlinegnat} Apply Algorithm \ref{alg:onlinegnat} to perform
 the GNAT ROM simulation for inputs $\muonline\in\mathcal D$ specified online.
 \item\label{step:postprocess} Apply Algorithm \ref{postProcess} to compute
 the desired outputs.
 \end{enumerate}

\begin{algorithm}[htbp]
\caption{Online step \ref{step:onlinegnat}: GNAT ROM simulation}
\begin{algorithmic}[1]\label{alg:onlinegnat}
\REQUIRE{online matrices $A$ and $B$, online inputs $\muonline\in\mathcal
D$, initial condition $\mask{\state^0(\muonline)}$, and state POD basis
$\mask{\podstate}$}
\ENSURE{generalized coordinates $\state_r^n(\muonline)$, $n=1,\ldots,\nt$}
\FOR {$n=0,\ldots,\nt-1$}
\STATE Choose initial guess $\statecoords^{n+1(0)}$ (e.g.,
$\statecoords^{n+1(0)}\leftarrow \statecoords^n$ with $\statecoords^0 = 0$).
\STATE $k\leftarrow 0$
\WHILE {not converged}
\STATE \label{step:resEval}Compute
\begin{align*}
C^{n(k)} &=\restrict{}\frac{\partial R^n}{\partial w}\left(\mask^T\mask
\initialCond +
\mask^T\mask\podstate\statecoords^{n+1{(k)}};\muonline\right)\mask^T\mask\podstate\\
D^{n(k)} &=\restrict{}R^n\left(\mask^T\mask
\initialCond +
\mask^T\mask\podstate\statecoords^{n+1{(k)}};\muonline\right).
\end{align*}
\STATE \label{step:linAlg}Compute
\begin{align*}
\pnk &= \arg\min_{\upsilon \in\RR{n_\state}} \|AC^{n(k)} \upsilon + BD^{n(k)}\|_2\\
\statecoords^{n+1(k+1)} &= \statecoords^{n+1(k)}+\alpha^{n+1(k)}\pnk,
\end{align*} 
where $\alpha^{n+1(k)}$ is computed via line search search or is set to 1.
\STATE $\mask{\tilde\state^{n+1(k+1)}}\leftarrow
\mask{\initialCond}+\mask{\podstate} \statecoords^{n+1(k+1)}$
\STATE $k\leftarrow k+1$
\ENDWHILE
\STATE Set $\state_r^{n+1}(\muonline)\leftarrow\state_r^{n+1(k)}$ and save it;
it will be used to compute outputs using Algorithm \ref{postProcess}.
\ENDFOR
\end{algorithmic}
\end{algorithm}

\begin{algorithm}[htb]
\caption{Online step \ref{step:postprocess}: computation of outputs}
\begin{algorithmic}\label{postProcess}
\REQUIRE{online inputs $\muonline\in\mathcal D$, generalized coordinates
$\statecoords^n(\muonline)$, $n=1,\ldots,\nt$, initial condition
$\restrictpost\state ^{0}(\muonline)$, and state POD basis $\restrictpost\podstate$}
\ENSURE{outputs $z$}
\FOR {$n=1,\ldots,\nt$}
\STATE $\restrictpost\tilde\state ^{n}\leftarrow\restrictpost\state ^0 +\restrictpost\podstate
\statecoords^n$
\ENDFOR
\STATE Compute $z= {G}(\restrictpost^T\restrictpost\state ^0,\ldots,
\restrictpost^T\restrictpost\tilde\state
^{\nt},\muonline)$
\end{algorithmic}
\end{algorithm}
\section{Error bounds}\label{sec:errBackEuler}

Here, error bounds are developed for any discrete nonlinear model-reduction
method assuming that time discretization is performed using the backward-Euler
scheme. These bounds highlight the advantages of the GNAT method, as it
minimizes components of these error bounds.  

When Eq.~\eqref{eq:ODE} is time discretized using the backward-Euler scheme, 
the residual corresponding to time step $n$,
input parameters $\params$, and the sequence of states computed by the
high-dimensional CFD
model $\staten$, $n=0,\ldots,\nt$ 
can be written as
\begin{gather}\label{eq:backEulerResidual}
R^n(\statenp;\params) = \statenp-\staten-\Delta t F(\statenp,\tnp;\params).
\end{gather}
\begin{proposition}
Assume $f:(\state,t;\params)\mapsto  \state - \Delta t F(\state,t;\params)$ satisfies the following inverse
Lipschitz continuity condition for the online input $\muonline\in\mathcal D$
\begin{equation}\label{eq:firstLipschitz}
\frac{\|f(\state,\tn;\muonline)-f(y,\tn;\muonline)\|}{\|\state-y\|} \geq
\varepsilon > 0,\quad \forall
n\in\{1,\ldots,\nt\}.
\end{equation}
Furthermore, assume that the high-dimensional CFD model employs the backward-Euler scheme for
time-integration and computes states $\staten$, $n=1,\ldots,\nt$ that satisfy
an absolute tolerance for the residual
 \begin{equation} 
\|R^n(\statenp;\muonline)\|\leq \epsilon_\mathrm{Newton},\quad \forall
n\in\{0,\ldots,\nt-1\}.
  \end{equation} 

Then, for any sequence of states $\stateApproxn$, $n=0,\ldots, \nt$ satisfying
$\stateApprox^0 = \state^0$, a global error bound for the approximation of the
state at the $n$-th time step  is given by
\begin{equation} \label{eq:boundBackEuler}
\boxed{\|\staten-\stateApproxn\| \leq
\sum_{k=1}^{n}a^{k}b_{n-k}\leq\sum_{k=1}^{n}a^{k}c_{n-k}\leq\sum_{k=1}^{n}a^{k}d_{n-k},}
 \end{equation} 
where
 \begin{align} \label{eq:cn}
 a &\equiv \sup_{n\in\{1,\ldots,\nt\}}\sup_{\state\neq
 y}\frac{\|\state-y\|}{\|f(\state,\tn;\muonline)-f(y,\tn;\muonline)\|}\nonumber\\
b_n&\equiv \epsilon_\mathrm{Newton} + \| \Rapprox^n(\stateApproxnp;\muonline)\|\nonumber\\
c_n&\equiv \epsilon_\mathrm{Newton} + \|
\proj\Rapprox^n(\stateApproxnp;\muonline)\| +
\|(I-\proj)\Rapprox^n(\stateApproxnp;\muonline)\|\\
 d_n &\equiv \epsilon_\mathrm{Newton} + \|\proj
 \Rapprox^n(\stateApproxnp;\muonline)\| +
 \|\mathsf R^{-1}\|\|\left(I - \mathbb P\right)
 \Rapprox^n(\stateApproxnp;\muonline)\|\nonumber\\
\proj &\equiv \podres\hatpodrespseudo\restrict{}\nonumber\\
\mathbb P&\equiv\podres\podres^T\nonumber\\
\restrict{\podres} &\equiv \mathsf Q\mathsf R \nonumber,
 \end{align} 
where $\mathsf Q\in
\RR{\nin\times \nR}$, $\mathsf R \in \RR{\nR\times\nR}$, and
$\Rapprox^n(\state;\params)=\state-\stateApproxn-\Delta t
 F(\state,\tn;\params)$. 
\end{proposition}

\ref{app:EB} provides a proof of the above error bounds. Their consequences include:
 \begin{itemize} 
 \item Justification for the minimum-residual approach taken by the tier $\mathrm{II}$ Petrov--Galerkin ROM. Namely, by computing $\stateApproxnp = \arg\min\limits_{\bar
 \state\in \state^{n+1(0)}+\mathcal Y}\|\Rapprox^n(\bar\state;\muonline)\|$,
 the tier II Petrov--Galerkin ROM selects the element of the trial subspace that
 minimizes $b_n$, $n=1,\ldots,\nt$. This in turn minimizes the tightest error
 bound in \eqref{eq:boundBackEuler}.
 \item Justification for using $\podres = \podjac$ in GNAT (snapshot-collection procedures 0 and 1). In this case, the
 GNAT iterations are equivalent to applying the Gauss--Newton method for
 minimizing $\|\proj \Rapprox^n(\stateApproxnp;\muonline)\|$. As a result, GNAT computes
 $\stateApproxnp = \arg\min\limits_{\bar \state\in \state^{n+1(0)}+\mathcal
 Y}\|\proj\Rapprox^n(\bar\state;\muonline)\|$, which is the element of the trial
 subspace that minimizes the second term of both $c_n$ and $d_n$, $n=1,\ldots,\nt$.
 \item Justification for computing $\podres$ via POD. When computed by POD,
 the basis $\podres$ is the orthogonal basis of dimension $\nR$ that minimizes
 the average projection error over the set of residual snapshots; this
 projection error appears as the last term of $d_n$.
 \item The tightest bound in \eqref{eq:boundBackEuler} is computable by
 the tier $\mathrm {II}$ Petrov--Galerkin ROM if the Lipschitz constant $a$ can be computed or
 estimated.  This is due to the computability of $b_n$: it requires only the
 tolerance $\epsilon_\mathrm{Newton}$ and the residual norm at each time step.
 \item The bound $ \sum\limits_{k=1}^{n}a^{k}c_{n-k}$ (resp.\ $\sum\limits_{k=1}^{n}a^{k}d_{n-k}$) is computable by
 GNAT if the Lipschitz constant $a$ can be computed or estimated and the
 projection error $\|(I-\proj)\Rapprox^n(\stateApproxnp;\muonline)\|$ (resp.\ $\|(I-\mathbb P)\Rapprox^n(\stateApproxnp;\muonline)\|$) can be computed or estimated.
 This is due to the computability of $\|\proj
 \Rapprox^n(\stateApproxnp;\muonline)\|=\|\hatpodrespseudo\restrict{\Rapprox^n(\stateApproxnp;\muonline)}\|$.
 The projection error $\|(I-\proj)\Rapprox^n(\stateApproxnp;\muonline)\|$ can be estimated by
 \begin{align} \label{eq:PE}
   \|(I-\proj)\Rapprox^n(\stateApproxnp;\muonline)\|&\approx\|(\podresp\hatpodrespseudop -
   \podres\hatpodrespseudo)\restrict{\Rapprox^n(\stateApproxnp;\muonline)}\|\\
	  &= \left\|\left(\hatpodrespseudop -
   \left[\begin{array}{c}
	 \hatpodrespseudo \nonumber\\
	 0_{\left(\nRp - \nR\right)\times \nin}
	 \end{array}\right]
	 \right)\restrict{\Rapprox^n(\stateApproxnp;\muonline)}\right\| ,
  \end{align} 
 where $\podresp\equiv \matvec{\phi_R}{\nRp}$ for some $\nRp>\nR$
 (following Ref.\ \citep{drohmannEOI}). Alternatively, the
 projection error $\|(I-\mathbb P)\Rapprox^n(\stateApproxnp;\muonline)\|$ can be
 approximated by the sum of the squares of the singular values neglected by
 $\podres$ (following Ref.\ \citep{chatErr}).
 \end{itemize}

\section{Implementation of online computations and post-processing} \label{sec:implementation}

\subsection{Sample mesh concept}\label{samplemesh} A quick inspection of
Algorithm \ref{alg:onlinegnat} reveals that the online stage of the GNAT
method performs CFD computations in only Step \ref{step:resEval} of this
algorithm. Furthermore,  these computations require, manipulate, and generate
information only at the nodes of the CFD mesh associated with the index sets
$\mathcal I$ and $\mathcal J$ described in
Section~\ref{sec:sysApproxOptimality}. Similarly, an inspection of Algorithm
\ref{postProcess} for output computation reveals that this algorithm requires
online access only to information pertaining to the nodes of the CFD mesh
associated with the index set $\mathcal K$ described in Section
\ref{sec:outputComputation}.  From these two observations, it follows that
both online algorithms can be effectively implemented by constructing in each
case a `sample mesh' tailored to the computations to be performed. This
concept is related to the subgrid idea recently proposed
in~\citep{haasdonkExplicit,drohmann2009reduced}, but differs from it primarily
in the node sampling algorithm as discussed in Section~\ref{sec:USA}.

For the purpose of building and exploiting the GNAT ROM, the sample mesh used
to execute Algorithm \ref{alg:onlinegnat} must
contain only the nodes associated with index sets $\mathcal I$ and
$\mathcal J$, and the geometrical entities (e.g., edges, faces, cells)
associated with these nodes.  For example, consider the case where the flow
solver of interest operates on unstructured tetrahedral meshes and is based on
a second-order finite-volume spatial discretization, where the discrete
unknowns are located at the mesh nodes and the fluxes are computed across the
boundaries of control volumes.  Here, the control volumes (or dual cells) are
constructed by connecting the centroids of the triangular faces and the
midpoints of the edges. The union of these control volumes is often referred
to as the dual CFD mesh. In this case, the sample mesh is constructed by
assembling the nodes associated with the index sets $\mathcal I$ and
$\mathcal J$, and the edges, faces, and cells required to allow the same CFD
solver to perform the computations in Step \ref{step:resEval} of Algorithm
\ref{alg:onlinegnat} as if it were operating on the original CFD mesh.
Figure~\ref{fig:reducedEval} illustrates this case in two dimensions. Note
that although the sample mesh shown in this figure is simply connected, this
property is not required. The reason for this is that the sample mesh has no specific
geometrical meaning and therefore no connectivity requirement aside from that
imposed by the stencil of the flow solver's spatial discretization scheme.

\begin{figure}[htbp]
\centering
\includegraphics[width=10cm]{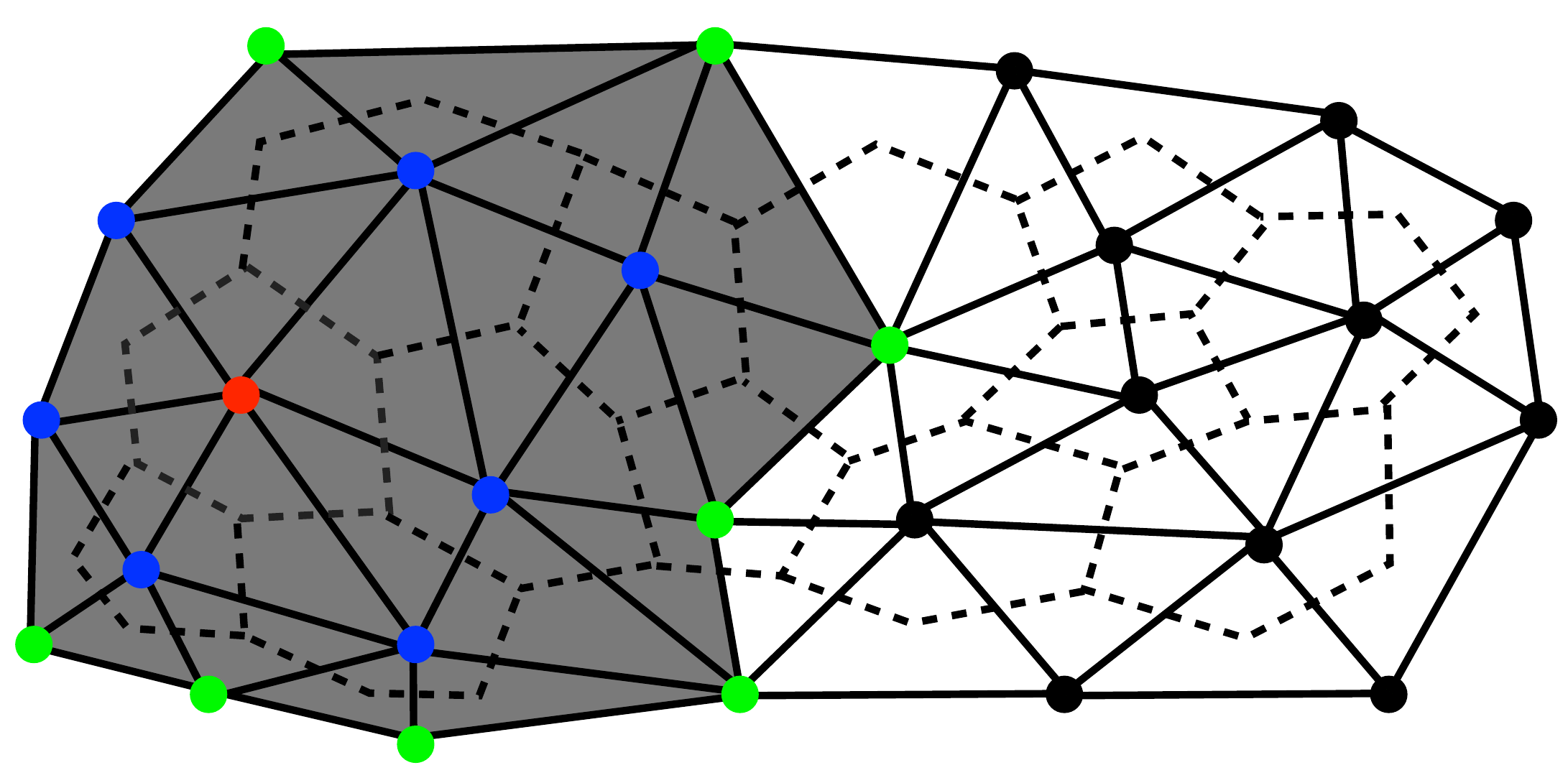}
\caption{Original and sample CFD meshes for online GNAT computations based on a second-order finite volume method operating on a dual mesh.
The original mesh is defined by all triangles, control volumes (dashed lines), edges, and nodes. The sample
mesh is defined by the gray-shaded triangles, associated edges, all control volumes fully
contained within the gray region, and the red, blue, and green nodes.
The residual and its Jacobian are computed at the red node (which defines
$\mathcal I$) , and the fluid state vector is computed at the red, blue, and
green nodes (which together define $\mathcal J$).}
\label{fig:reducedEval}
\end{figure}

The second sample mesh used to execute Algorithm \ref{postProcess} for output
computation must contain only the nodes associated with the index set
$\mathcal K$, and their corresponding geometrical entities. For example, when
the desired outputs are the lift and drag, the sample mesh is the wet surface
mesh --- that is, the collection of faces and nodes lying on the obstacle
around which the flow is computed --- and any other geometrical entity the
flow solver requires to compute the outputs.

As previously mentioned, one major advantage of the sample-mesh approach is
that it allows the same flow solver that was used for offline,
high-dimensional CFD computations to be used for online, low-dimensional
computations.  Consequently, the online ROM computations can be automatically
distributed and parallelized in the same manner as their large-scale CFD
counterparts. The ROM's performance can also be expected to scale (in the weak
sense, or the scaled speedup metric) in a similar manner to that of the
large-scale flow computations using the same CFD solver. However, because its
size is typically a \emph{tiny fraction} of that of the original CFD mesh, the
sample mesh can be expected to require significantly fewer computational cores
and lead to simulations requiring far fewer core-hours than its large-scale
counterpart; Section \ref{sec:ahmedFine} demonstrates this.

Determining the index set $\mathcal K$ is a trivial task. For this reason,
constructing the sample mesh for output computation is also straightforward.
However, determining the sample-index set $\mathcal I$ at which to compute the
residual and its Jacobian is a more delicate matter that is discussed next.

\subsection{Underlying node sampling algorithm}\label{sec:sampleNodeSelection}
\label{sec:USA}

Several algorithms have been proposed in the literature for selecting the
sample indices that define the sample matrix $\restrict{}$.
Usually, these algorithms are tailored to the hyper-reduction procedure they
are designed to support. For example, for the various forms of empirical
interpolation outlined in Section~\ref{sec:intro}, several algorithms for
selecting sample indices have been developed around the objective of
minimizing the error in the interpolated
snapshots~\citep{barrault2004eim,chaturantabut2010journal,haasdonkExplicit},
the difference between the interpolated snapshots and their orthogonal
projections onto the subspace of
approximation~\citep{nguyen2008bpi,galbally2009non}, and the condition number
of the normal-equations matrix used for interpolation or least-squares
approximation \citep{willcox2006ufs, astrid2007mpe}. However, because these
algorithms are sensitive to differences in scale between different
conservation equations, they are not particularly suitable for CFD
applications, as they would lead to a biased treatment of the multiple
conservation equations defined at each mesh node. For this reason,
Ref.~\citep{drohmann2Phase} applied the greedy sampling algorithm adopted in
Refs.~\citep{barrault2004eim,chaturantabut2010journal,haasdonkExplicit} to
each conservation equation separately.  However, this approach causes the
conservation equations to be sampled at different sets of nodes.  This not
only complicates the implementation of online CFD computations, it also leads
to a larger subgrid than necessary and therefore to computationally suboptimal
nonlinear ROM simulations.

Here, Algorithm \ref{greedy} is proposed for determing the sample nodes from a
given CFD mesh, and therefore constructing the sample-index set $\mathcal I$
and sample matrix $\restrict{}$.  This algorithm is based on the greedy method
presented in
Refs.~\citep{barrault2004eim,chaturantabut2010journal,haasdonkExplicit}. This choice
is made because this method attempts to minimize the error
$\|(I-\proj)\Rapprox^n(\stateApproxnp;\muonline)\|$ associated
with the gappy POD projection of the residual, and therefore the third term of
the coefficient $c_n$ \eqref{eq:cn} characterizing the error
bound~\eqref{eq:boundBackEuler}. However, Algorithm \ref{greedy} distinguishes
itself from the method described
in Refs.~\citep{barrault2004eim,chaturantabut2010journal,haasdonkExplicit} in a few
noteworthy aspects. First, it allows for overdetermined least-squares matrices
as opposed to relying on interpolation of the nonlinear function.  Secondly,
it allows different bases to be used to approximate the residual and its Jacobian.
Finally, it operates directly on the mesh nodes instead of the algebraic
indices. Thus, the sample-index set $\mathcal I$ consists of the 
degrees of freedom associated with the nodes in set $\mathcal
N\equiv\{\mathsf n_1, \mathsf n_2,\ldots, \mathsf n_{\nnode}\}$. Because of
this latter minor albeit distinctive feature,  Algorithm \ref{greedy} treats
all conservation equations in a balanced manner, does not lead to a
larger-than-necessary sample mesh, and therefore does not introduce from the outset a
computational inefficiency in the online ROM computations. 

\begin{remark}{\bf 2}
The sample-node set $\mathcal N$ can be seeded with nodes of the CFD mesh that
are deemed important due to their strategic locations.  In particular, it
is essential that at least one sample node lies on the inlet or outlet
boundary of the problem, if such a boundary exists.  It is equally essential
that each input variable $\params_i$, $i=1,\ldots,d$ affects the value of the
residual at at least one sample node.  If the above conditions are not met,
the hyper-reduced GNAT model will be blind to the boundary conditions and inputs
\citep{astrid2004rps}.
\end{remark}

\begin{algorithm}[htbp]
\caption{Greedy algorithm for selecting sample nodes from a given CFD mesh}
\label{greedy}
\begin{algorithmic}[1]
\REQUIRE{$\podres$; $\podjac$; target number of sample nodes $\nnode$; seeded
sample-node set $\mathcal N$ (see {\bf Remark 2}); number of working columns
of $\podres$ and $\podjac$ denoted by $\nbasisgreed \leq \min(\nR, \nJ,
\neqpernode\nnode)$, where $\neqpernode$ denotes the number of unknowns at a
node (for example, $\neqpernode = 5$ for three-dimensional
compressible flows without a turbulence model).}
\ENSURE{sample-node set $\mathcal N$ }
\STATE Compute the additional number of nodes to sample: $\nnodeadd = \nnode -
|\mathcal N|$
\STATE Initialize counter for the number of working basis vectors used:
$\Qtot\leftarrow 0$
\STATE Set the number of greedy iterations to perform: $\niter = \min(\nbasisgreed,\nnodeadd)$ 
\STATE 
Compute the maximum number of right-hand sides in the least-squares problems: $\nrhsmax =\mathrm{ceil}(\nbasisgreed /\nnodeadd)$
\STATE Compute the minimum number of working basis vectors per iteration:
$\nbasisgreedpitmin = \mathrm{floor} (\nbasisgreed/\niter)$
\STATE Compute the minimum number of sample nodes to add per iteration:
$\nnodeaddpitmin = \mathrm{floor} (\nnodeadd\nrhsmax/\nbasisgreed)$
\FOR[greedy iteration loop]{$\iter =1,\ldots,\niter$}
\STATE Compute the number of working basis vectors for this iteration: $\nbasisgreedpit \leftarrow
\nbasisgreedpitmin$; \\ if ($\iter \leq \nbasisgreed\ \mathrm{mod}\
\niter$),
then $\nbasisgreedpit\leftarrow \nbasisgreedpit+1$
\STATE 
Compute the number of sample nodes to add during this iteration: $\nnodeaddpit\leftarrow
\nnodeaddpitmin$;
\\ if $(\nrhsmax =1)$ and  $(\iter \leq \nnodeadd \ \mathrm{mod}\
\nbasisgreed)$, then $\nnodeaddpit\leftarrow \nnodeaddpit+1$ 
\IF{$\iter=1$}
\STATE $\vecmat{\mathsf R}{\nbasisgreedpit} \leftarrow \vecmat{\phi_R}{\nbasisgreedpit}$
\STATE $\vecmat{\mathsf J}{\nbasisgreedpit} \leftarrow \vecmat{\phi_J}{\nbasisgreedpit}$
\ELSE
\FOR[basis vector loop]{$q=1,\ldots, \nbasisgreedpit$}
\STATE\label{step:resRecon} $\mathsf R^q\leftarrow \phi_R^{\Qtot+q}-
\vecmat{\phi_R}{\Qtot}\alpha$, with $\alpha
=\arg\min\limits_{\gamma\in\RR{\Qtot}}\Big\|{\vecmat{\restrict{\phi_R}}{\Qtot}}\gamma -
\restrict{\phi}_R^{\Qtot+q}\Big\|_2$
\STATE\label{step:jacRecon} $\mathsf J^q\leftarrow \phi_J^{\Qtot+q}-
\vecmat{\phi_J}{\Qtot}\beta$, with $\beta
=\arg\min\limits_{\gamma\in\RR{\Qtot}}\Big\|{\vecmat{\restrict\phi_J}{\Qtot}}\gamma -
\restrict{\phi}_J^{\Qtot+q}\Big\|_2$
\ENDFOR

\ENDIF
\FOR[sample node loop]{$j=1,\ldots, \nnodeaddpit$}
\STATE\label{step:greedy1} Choose node with largest average error:
$n\leftarrow \arg\max\limits_{l \notin \mathcal
N}\sum\limits_{q=1}^{\nbasisgreedpit}\left(\sum\limits_{i \in
\dofs{l}}\left((\mathsf R_i^q)^2+ (\mathsf J_i^q)^2\right)\right)$,\\
where $\dofs{l}$ denotes the degrees of freedom associated with node $l$.
\STATE $\mathcal N\leftarrow \mathcal N \cup \{n\}$
\ENDFOR
\STATE $\Qtot\leftarrow \Qtot + \nbasisgreedpit$
\ENDFOR
\end{algorithmic}
\end{algorithm}
\section{Applications}\label{sec:experiments}

To illustrate the ability of GNAT to reduce the dimension and complexity of
highly nonlinear CFD models while maintaining a high level of accuracy, this
section considers two examples.  The first one is an academic problem based on
Burgers' equation. It features a moving shock, and therefore highlights GNAT's
potential for unsteady CFD problems with moving discontinuities. In this
one-dimensional example, GNAT is applied in a prediction scenario --- that is,
for the (most relevant) case where the values of the input variables change
	between the offline training simulations and the online simulation. The
	second example pertains to the computation of the Ahmed body wake
	flow~\citep{ahmed}, which is a well-known CFD benchmark problem in the
	automotive industry. The CFD model employed for this three-dimensional
	problem is characterized by millions of unknowns and therefore incurs
	time-consuming offline computations. For this reason, GNAT is applied in
	this example in reproduction mode only --- that is, for the (preliminary)
	scenario where the online input-variable values are identical to their
	training counterparts.  Nevertheless, this example demonstrates GNAT's
	performance on a realistic, large-scale turbulent flow problem, and
	contrasts it with that of other nonlinear model-reduction methods.

\subsection{Parameterized inviscid Burgers' equation}\label{sec:burgers}

This numerical experiment employs the problem setup described in Ref.~\citep{tpwl}.
Consider the parameterized initial boundary value problem (IBVP)
\begin{align}
\label{eq:burgers}
\frac{\partial U(x,t)}{\partial t} + \frac{1}{2}\frac{\partial \left(U^2\left(x,t\right)\right)}{\partial x} &= 0.02e^{bx}\\
U(0,t) &= a, \ \forall t>0\\
\label{eq:burgersLast}U(x,0) &= 1, \ \forall x\in\left[0,~100\right],
\end{align}
where $a$ and $b$ are two real-valued input variables. This problem is
discretized using Godunov's scheme, which leads to a finite-volume
formulation. The one-dimensional domain is discretized using a grid with 4001
nodes corresponding to coordinates coordinates $x_i = i \times (100/4000)$,
$i=0,\ldots, 4000$.  Hence, the resulting CFD model is of dimension $N=4000$.
The solution $U(x,t)$ is computed in the time interval
$t\in\left[0,~4000\right]$ using a uniform computational time-step size
$\Delta t = 0.05$, leading to $\nt = 1000$ total time steps.  Because there is
only one unknown per node, each sample node corresponds to a single sample
index.

First, a GNAT model is constructed using snapshot-collection procedure 2 (see
Table~\ref{table:snapMethods}) and the following parameters: $\nstate = 50$,
$\nR = 160$, $\nJ = 70$ and $\nin = 160$. It is trained for the solution of
the IBVP~(\ref{eq:burgers})--(\ref{eq:burgersLast}) using the the values of
the boundary-condition parameter $a$ and source-term parameter $b$ reported in
columns 2--4 of Table~\ref{tab:burgersGlobalInputs}. 

\begin{table}[tb] 
	\caption{Offline and online inputs for the
	IBVP~\eqref{eq:burgers}--\eqref{eq:burgersLast}}
  \label{tab:burgersGlobalInputs} 
 \centering 
 \begin{tabular}{||c||c|c|c|c||} 
  \hline 
\multirow{2}{*}{Input variables} & Training input \#1& Training input \#2& Training input \#3& Online input\\
                                    &$\params^1$        &$\params^2$        &$\params^3$        &$\muonline$\\
  \hline 
$a$ & 3 & 6 &9 & 4.5 \\
$b$ & 0.02& 0.05 &0.075 &  0.038\\
  \hline 
  \end{tabular} 
  \end{table} 

Next, the resulting global GNAT ROM is applied online to the solution of the
IBVP~(\ref{eq:burgers})--\ref{eq:burgersLast}) configured with the new values
of $a$ and $b$ shown in column 5 of Table~\ref{tab:burgersGlobalInputs}. A
reference solution for this problem is also computed using the
high-dimensional CFD model. Both solutions are graphically depicted in
Figure~\ref{fig:burgersGlobal} and were computed using a single processor.

The reader can observe that the GNAT prediction closely matches the reference
solution in general. Although oscillations in the GNAT solution are apparent
at $t=2.5$, they dissipate over time.  The relative time-averaged discrepancy
between the GNAT solution and the reference high-dimensional CFD solution as
measured in the Euclidean norm of the state vector is only 1.26\%.  The
high-dimensional CFD solution took 1167 times longer to complete than the
online GNAT solution; this showcases the improved CPU performance delivered by
the GNAT ROM.

\begin{figure}[htbp] \centering
\includegraphics[width=0.6\textwidth]{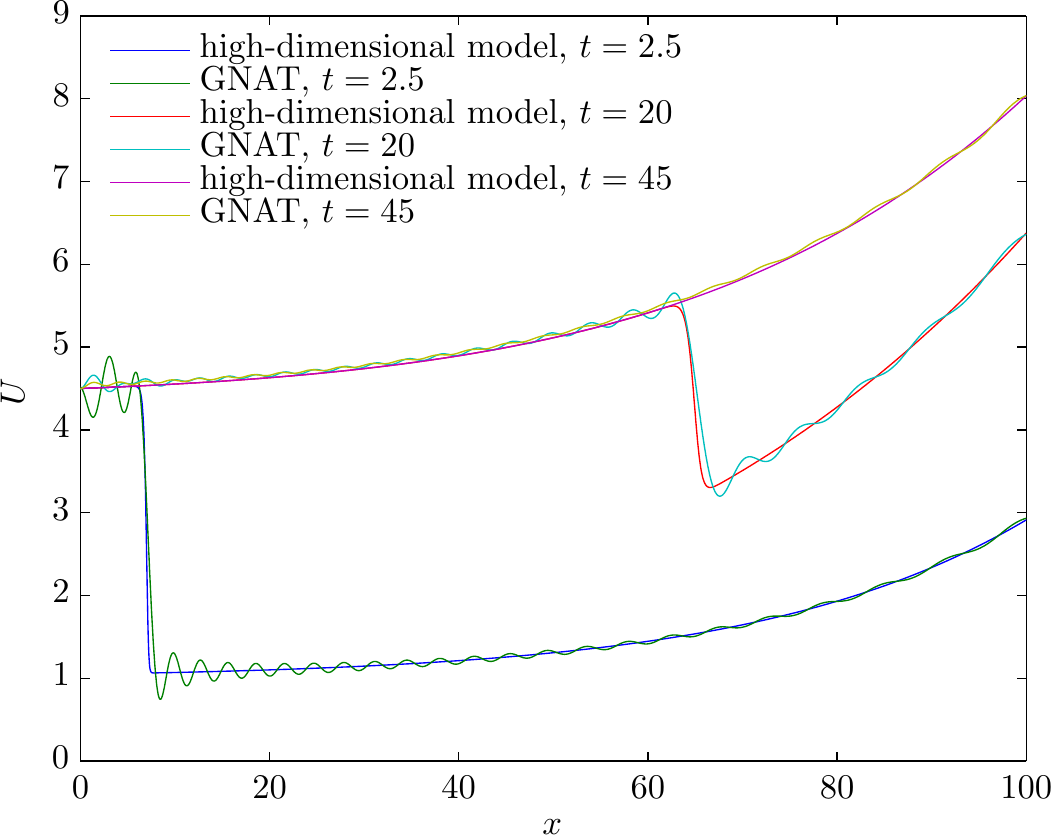} 
\caption{Performance of GNAT in predictive mode for the
IBVP~\eqref{eq:burgers}--\eqref{eq:burgersLast}}
\label{fig:burgersGlobal} \end{figure}

\subsection{Ahmed-body wake flow}\label{sec:ahmedFine}

\subsubsection{Preliminaries}

To assess GNAT's performance on large-scale CFD applications, GNAT is
implemented in the massively parallel compressible-flow solver
AERO-F~\citep{geuzaine2003aeroelastic,farhat2003application}.  For turbulent,
viscous flow computations, this finite-volume CFD code offers various RANS and
LES turbulence models, as well as a wall function.  It performs a second-order
semi-discretization of the convective fluxes using a method based on the Roe,
HLLE, or HLLC upwind scheme. It can also perform second- and fourth-order
explicit and implicit temporal discretizations using a variety of time
integrators. The GNAT implementation in AERO-F is characterized by the
sample-mesh concept described in Section \ref{sec:implementation}.  All linear
least-squares problems and singular value decompositions are computed in
parallel using the ScaLAPACK library \citep{scalapack}.  AERO-F is used here
to demonstrate GNAT's potential when applied to a realistic, large-scale,
nonlinear benchmark CFD problem: turbulent flow around the Ahmed body.

The Ahmed-body geometry~\citep{ahmed} is a simplied car geometry.  It can be
described as a modified parallelepiped featuring round corners at the front
end and a slanted face at the rear end (see Figure~\ref{fig:ahmed}). Depending
on the inclination of this face, different flow characteristics and wake
structure may be observed. For a slant angle $\varphi \ge 30^{\circ}$, the
flow features a large detachment. For smaller slant angles, the flow
reattaches on the slant.  Consequently, the drag coefficient suddenly
decreases when the slant angle is increased beyond its critical value of
$\varphi = 30^{\circ}$. Due to this phenomenon, predicting the flow past the
Ahmed body for varying slant angles has become a popular benchmark in the
automotive industry.

\begin{figure}[btp]
\centering
\includegraphics[width=.8\textwidth]{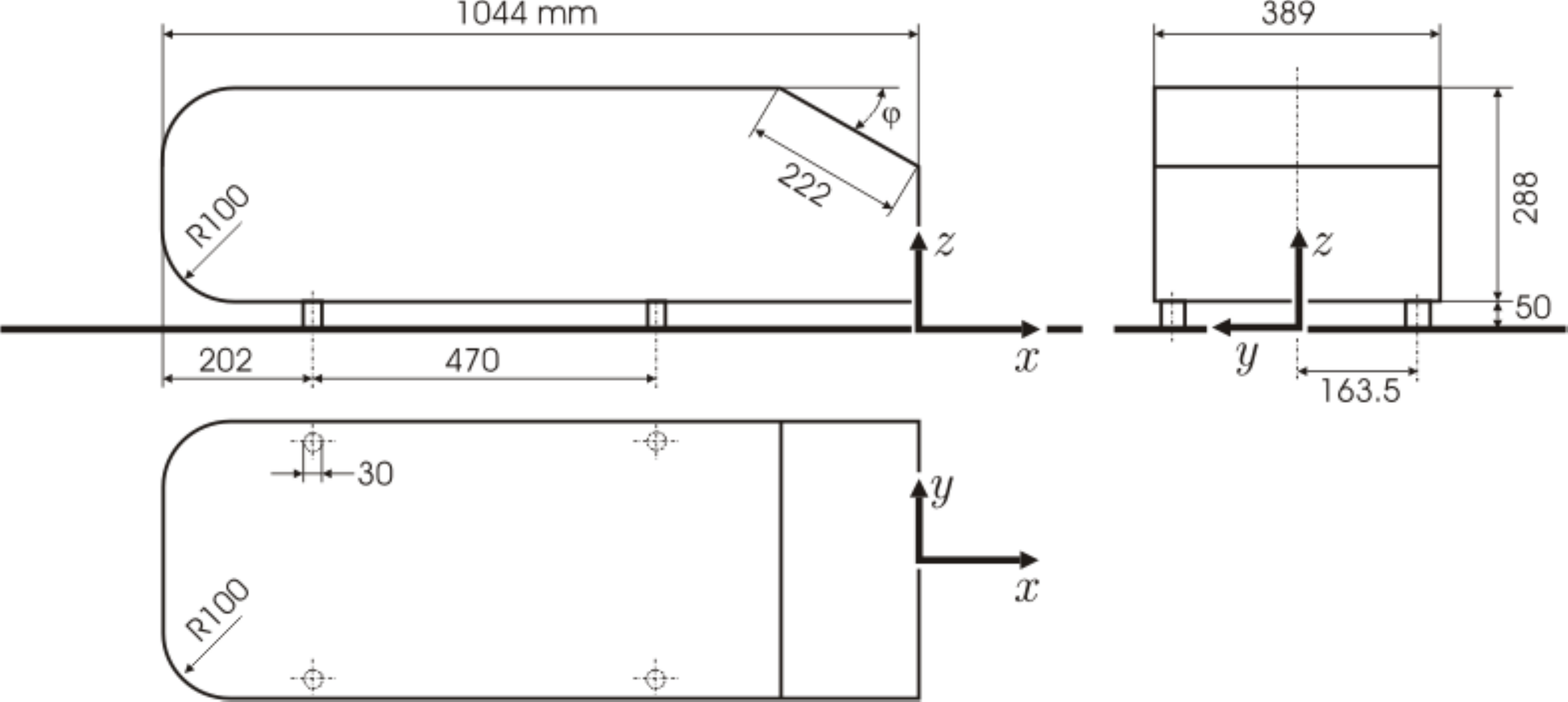}
\caption{Geometry of the Ahmed body (from Ref.~\citep{hinterberger2004large})}
\label{fig:ahmed}
\end{figure}

This work considers the subcritical angle $\varphi = 20^{\circ}$ and treats
the drag coefficient $C_D=\frac{D}{\frac{1}{2}\rho_\infty V_\infty^2
5.6016\times 10^{-2}~\mathrm{m}^2}$ around the body as the output of interest.
The free-stream velocity is set to $V_\infty=60$ m/s, and the Reynolds number
based on a reference length of 1.0 m is set to $\mathrm{Re} = 4.29\times
10^6$.  The free-stream angle of attack is set to $0^{\circ}$.

\subsubsection{High-dimensional CFD model}
The high-dimensional CFD model corresponds to an unsteady Navier--Stokes
simulation using AERO-F's DES turbulence model and wall function.  The fluid
domain is discretized by a mesh with 2,890,434 nodes and 17,017,090 tetrahedra
(Figure~\ref{fig:ahmedFineMesh}). A symmetry plane is employed to exploit the
symmetry of the body about the $x$--$z$ plane.  Due to the turbulence model
and three-dimensional domain, the number of conservation equations per node is
$\neqpernode = 6$, and therefore the dimension of the CFD model is $N =
17,342,604$. Roe's scheme is employed to discretize the convective fluxes; a
linear variation of the solution is assumed within each control volume, which
leads to a second-order space-accurate scheme.

\begin{figure}[htbp]
\centering
\includegraphics[width=.7\textwidth]{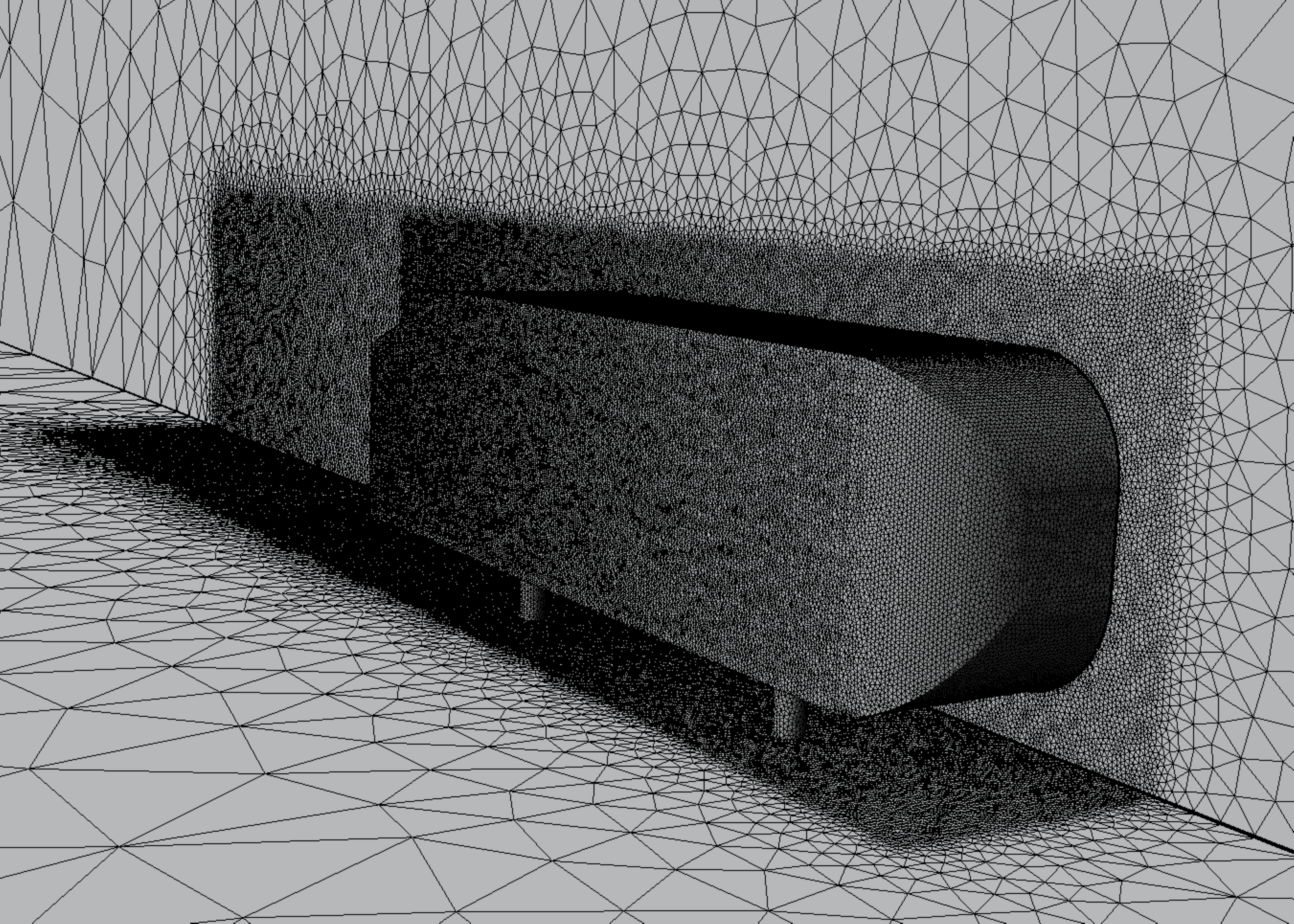}
\caption{CFD mesh with 2,890,434 grid points and 17,017,090 tetrahedra (partial
view,  $\varphi=20^{\circ}$). Darker areas indicate a more refined area of the
mesh.}
\label{fig:ahmedFineMesh}
\end{figure}

Flow simulations are performed within a time interval $t\in\left[0\
\hbox{s},~0.1\ \hbox{s}\right]$, the second-order accurate implicit
three-point backward difference scheme is used for time integration, and the
computational time-step size is fixed to $\Delta t = 8\times 10^{-5}$ s. For the
chosen CFD mesh, this time-step size corresponds to a maximum CFL number of roughly
2000.  The nonlinear system of algebraic equations arising at each time step
is solved by Newton's method. Convergence is declared at the $k$-th iteration
for the $n$-th time step when the residual satisfies $\|R^{n(k)}\| \leq 0.001
\|R^{n(0)}\|$.  All flow computations are performed in a non-dimensional setting.

A steady-state simulation computes the initial condition for the unsteady
simulation. This steady-state calculation is characterized by the same
parameters as above, except that it employs local time stepping with a maximum
CFL number of 50, it uses the first-order implicit backward Euler time
integration scheme, and it employs only one Newton iteration per (pseudo) time
step.

All computations are performed in double-precision
arithmetic on a parallel Linux cluster\footnote{The cluster contains
compute nodes with 16 GB of memory. Each node consists of two quad-core Intel
Xeon E5345 processors running at 2.33 GHz inside a DELL Poweredge 1950. The
interconnect is Cisco DDR InfiniBand.} using a variable number of
cores.  

\subsubsection{Comparison with experiment}\label{sec:validation}

Ref.\ \citep{ahmed} reports an experimental drag coefficient of 0.250 around
the Ahmed body for a slant angle of $\varphi = 20^\circ$. Figure
\ref{fig:fomResponse} reports the time history of the drag coefficient
computed using the high-dimensional CFD model described in the previous
section.  Indeed, the time-averaged value of the computed drag coefficient
obtained using the trapezoidal rule is $C_D = 0.2524$. Hence, it is within
less than 1\% of the reported experimental value.  This asserts the quality of
the constructed CFD model and AERO-F's computations.  For reference, this
high-dimensional CFD simulation consumed $13.28$ hours on 512 cores.

\begin{figure}[tbp] \centering
\includegraphics[width=0.8\textwidth]{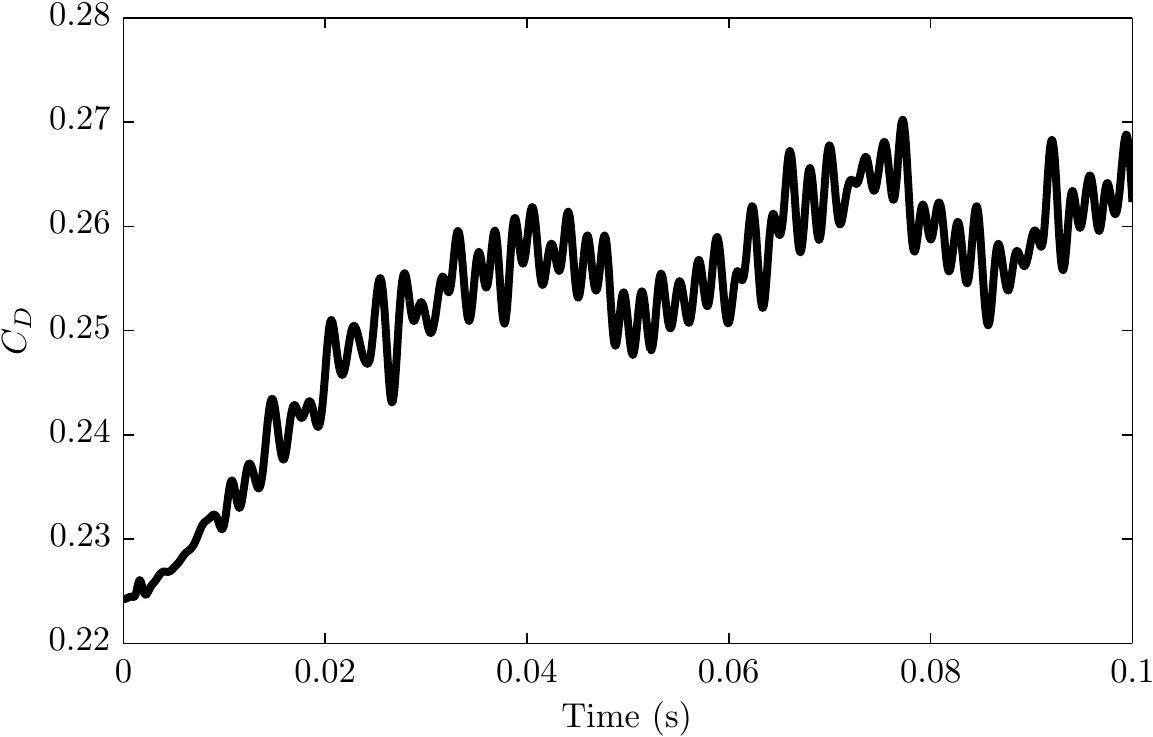} 
\caption{{Time history of the drag coefficient predicted for $\varphi =
20^{\circ}$ using DES and a CFD mesh with $N = 17,342,604$ unknowns. Oscillatory
behavior due to vortex shedding is apparent.}}
\label{fig:fomResponse} \end{figure}

\subsubsection{ROM performance metrics}
The following metrics will be used to assess GNAT's performance. The relative
discrepancy in the drag coefficient, which assesses the accuracy of a GNAT
simulation, is measured as follows:
 \begin{equation} \label{eq:ahmedRelError}
\mathrm{RD} = \frac{\frac{1}{\nt}\sum\limits_{n=1}^{\nt}|{C_D}_{\mathrm{I}}^n-{C_D}_{\mathrm{III}}^n|}{\max\limits_n {C_D}_{\mathrm{I}}^n - \min\limits_n {C_D}_{\mathrm{I}}^n},
 \end{equation} 
where ${C_D}_{\mathrm{I}}^n$ denotes the drag coefficient computed at the
$n$-th time step using the high-dimensional CFD model (tier I model), and
${C_D}_{\mathrm{III}}^n$ denotes the corresponding value computed using the
GNAT ROM (tier III model).

The improvement in CPU performance delivered by GNAT as measured in wall time is defined as 
\begin{equation} \label{eq:compSpeedup}
\mathrm{\WT} = \frac{T_{\mathrm{I}}}{T_{\mathrm{III}}},
\end{equation}
where $T_{\mathrm{I}}$ denotes the wall time consumed by a flow simulation
associated with the high-dimensional CFD model, and $T_{\mathrm{III}}$ denotes
the wall time consumed {\it online} by its counterpart based on a GNAT ROM.
For the high-dimensional model, the reported wall time includes the solution
of the governing equations and the output of the state vector; for the GNAT
reduced-order model, it includes the execution of Algorithm
\ref{alg:onlinegnat}.  After the completion of Algorithm \ref{alg:onlinegnat},
Algorithm \ref{postProcess} is executed to compute the drag coefficient. This
output-computation step employs a sample mesh based on nodes $\mathcal K$
determined from the wet surface; it is characterized by 124,047 nodes and
492,445 tetrahedral cells. For all reduced-order models, Algorithm
\ref{postProcess} consumed 12.2 minutes on 4 cores, or 9.7 minutes on 8 cores.

The improvement in CPU performance delivered by GNAT as measured in computational resources is
defined as
\begin{equation} \label{eq:compSpeedup}
\mathrm{\RS} = \frac{c_{\mathrm{I}}T_{\mathrm{I}}}{c_{\mathrm{III}}T_{\mathrm{III}}},
\end{equation}
where  $c_{\mathrm{I}}$ and $c_{\mathrm{III}}$ denote the number of cores
allocated to the high-dimensional and GNAT-ROM simulations, respectively.

As reported in Section \ref{sec:validation}, the high-dimensional CFD
simulation is characterized by $T_{\mathrm{I}} = 13.28$~hours and
$c_{\mathrm{I}} =512$~cores, which leads to  $c_{\mathrm{I}} T_{\mathrm{I}} =
6,798$~core-hours. 

\subsubsection{GNAT performance assessment}\label{sec:snapStudy}
This section assesses GNAT's performance for two different snapshot-collection
procedures: procedures 0 and 1 of Table \ref{table:snapMethods}.  Recall from
Section~\ref{sec:gapConsistency} that snapshot-collection procedure 0 is
inconsistent in the sense introduced in Ref.~\citep{carlbergGappy} and
restated in Section~\ref{sec:consistency}, but is similar to the approach most
often taken in the literature. Procedure 1 satisfies one consistency
condition. Procedure 2 is not tested because Ref.~\citep{carlbergThesis}
showed that it does not lead to robust reduced-order models; procedure 3 is
not tested due to computational infeasibility.

To build the state POD basis, consistent snapshots
$\{\state^n-\state^0\}_{n=1}^{\nt}$ with $\nt=1252$ are collected during
high-dimensional CFD simulation. Then, these snapshots are normalized to
prevent snapshots with large magnitudes from biasing the SVD.  The dimension of
the state POD basis
is set to $n_\state=283$, which corresponds to 99.99\% of
the total statistical energy of the (normalized) snapshots.\footnote{Numerical experiments reported
in Ref.~\citep{carlbergThesis} determined this to be an appropriate
criterion.}
All numerical studies carried out on the Ahmed body employ this POD
basis for the state.

Algorithm \ref{greedy} is employed to generate two sample meshes: one using
the matrix $\podres=\podjac$ generated by snapshot-collection procedure 0, and
one using $\podres=\podjac$ generated by snapshot-collection procedure
1.\footnote{Residual snapshots are normalized before $\podres$ is computed.}
This algorithm employs the following parameters: $\nbasisgreed= 219$, $\nnode
= 378$, and an initial sample-node set $\mathcal N$ seeded with the boundary
node whose entries of $\phi_R^1$ have the largest sum of squares.
Figure~\ref{fig:snap0snap1meshes} depicts the two resulting sample meshes.
Note that Algorithm~\ref{greedy} chooses sample nodes from three salient regions of
the computational fluid domain: the wake region behind the body, and the
region behind each cylindrical support.  This implies that
on average, the magnitude of the residual is highest in these regions during
the training simulations. This is consistent with the fact that the flow is
separated in these regions and is characterized by a strong vorticity.

\begin{figure}[htbp]
\centering
\subfigure[Sample mesh generated using snapshot-collection procedure 0]{
\includegraphics[width=.45\textwidth]{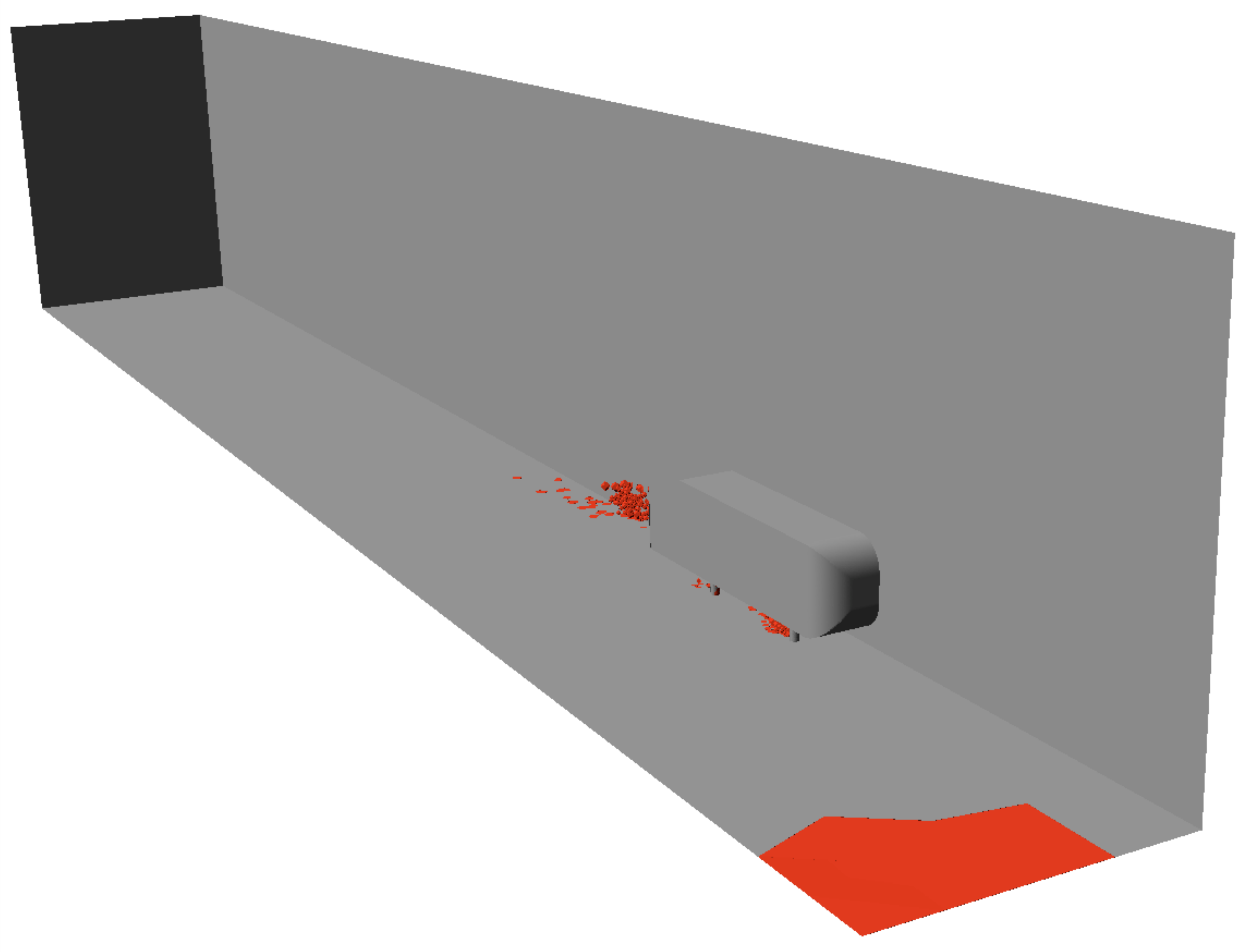}
}
\subfigure[Sample mesh generated using snapshot-collection procedure 1]{
\includegraphics[width=.45\textwidth]{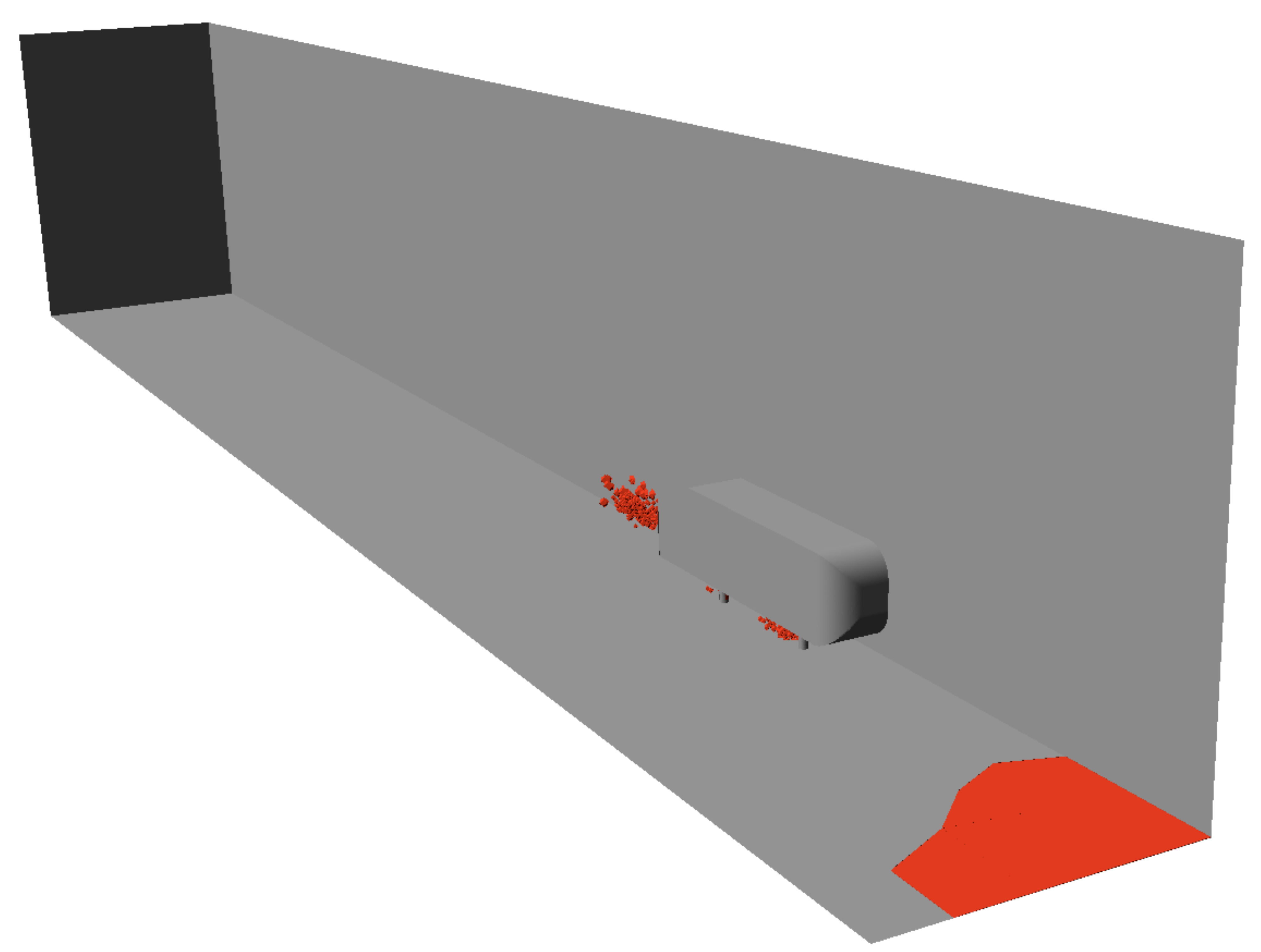}
\label{fig:snap1mesh378}
}
\caption{Sample meshes with 378 sample nodes generated by
Algorithm~\ref{greedy}. Sample meshes are shown in red, within the
computational fluid domain.}
\label{fig:snap0snap1meshes}
\end{figure}

The two GNAT models employ $\nJ=\nR=1514$; this corresponds to 99.99\% of the
energy in the snapshots of the residual collected during the tier
$\mathrm{II}$ ROM simulation.  Both GNAT simulations are executed using only 4
cores (as compared with 512 cores used for the high-dimensional model).  The
GNAT simulations employ the same Ahmed-body configuration and flow conditions
used for the high-dimensional CFD simulation. 

Figure~\ref{fig:snap0snap1} reports the time histories of the drag
coefficient predicted by the high-dimensional simulation and both GNAT ROM
computations. Figure~\ref{fig:pressureCompare} contrasts the surface pressure
contours at $t = 0.1$ s obtained using the high-dimensional model and the
GNAT ROM based on snapshot-collection procedure 1.
Table~\ref{tab:snap0snap1} provides the performance results for the ROM
simulations. These results demonstrate the following:
\begin{itemize} 
\item Both snapshot-collection procedures 0 and 1 lead to GNAT ROMs that
deliver improvement in CPU performance (as measured in computational
resources $\RS$) exceeding 230.  This occurs largely due to the drastic reduction in
cores made possible by the sample-mesh implementation, which allows the ROM
simulation to be executed on as few as 4 cores. In particular, the data
suggest that 438 parametric GNAT ROM simulations (using snapshot-collection
procedure 1) could be executed in a predictive scenario using the same
core-hours required by  a single high-dimensional CFD computation (see Table
\ref{tab:snap0snap1})--- a test
that will be conducted in the future.
\item When equipped with snapshot-collection procedure 1, which satisfies one
consistency condition,
the GNAT ROM reproduces almost perfectly the time history of the drag
coefficient computed by the high-dimensional simulation. On the other hand,
GNAT becomes less accurate
when equipped with snapshot-collection procedure 0, which is inconsistent (see
Figure \ref{fig:snap0snap1}).
Furthermore, GNAT requires fewer Newton iterations per time step for
convergence (and performs faster) when it is equipped with snapshot-collection
procedure 1 compared with 
snapshot-collection procedure 0 (see Table \ref{tab:snap0snap1}). These observations highlight the importance of the
consistency concept introduced during GNAT's development.
\item When equipped with snapshot-collection procedure 1, GNAT
delivers pressure-contour results that are almost identical to those computed
by the high-dimensional simulation, including in the wake region behind the
body where the flow is most complex (see Figure \ref{fig:pressureCompare}). 
\end{itemize}

\begin{figure}[htbp] \centering
\includegraphics[width=0.8\textwidth]{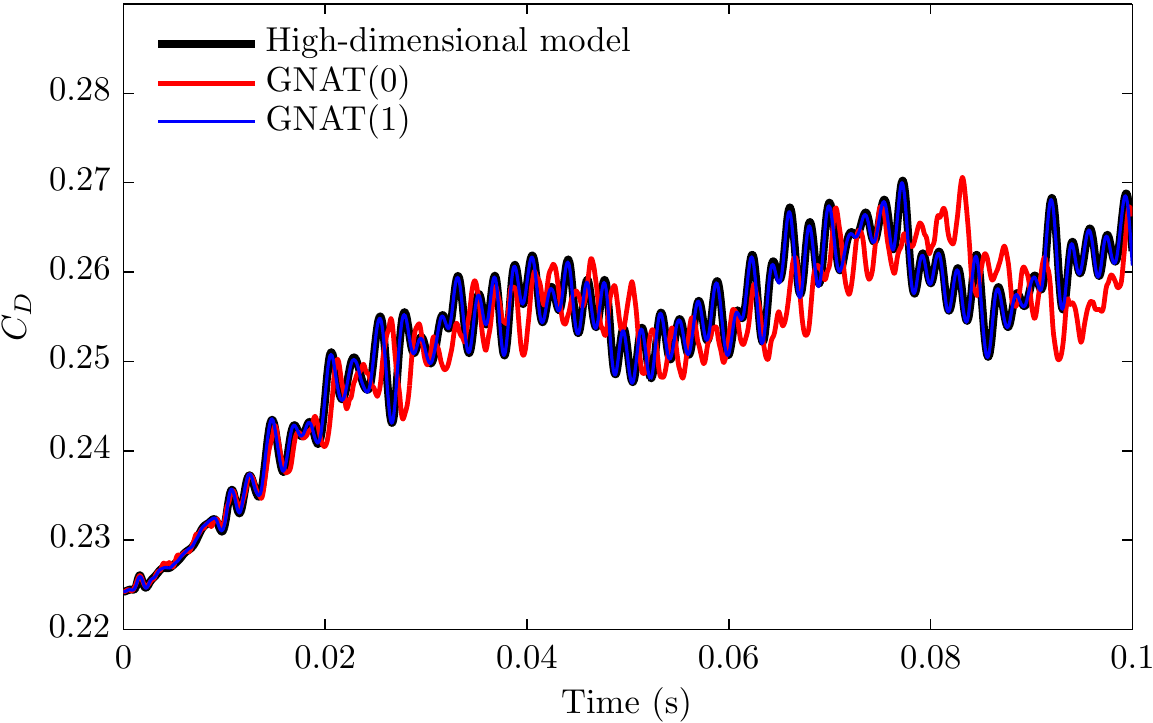} 
\caption{Computed time histories of the drag coefficient (GNAT($i$) refers to
GNAT equipped with snapshot-collection procedure $i$). GNAT(1) directly
overlays the high-dimensional model results.}
\label{fig:snap0snap1} \end{figure}

\begin{figure}[htbp]
\centering
\subfigure[High-dimensional CFD model]{
\includegraphics[width=.45\textwidth]{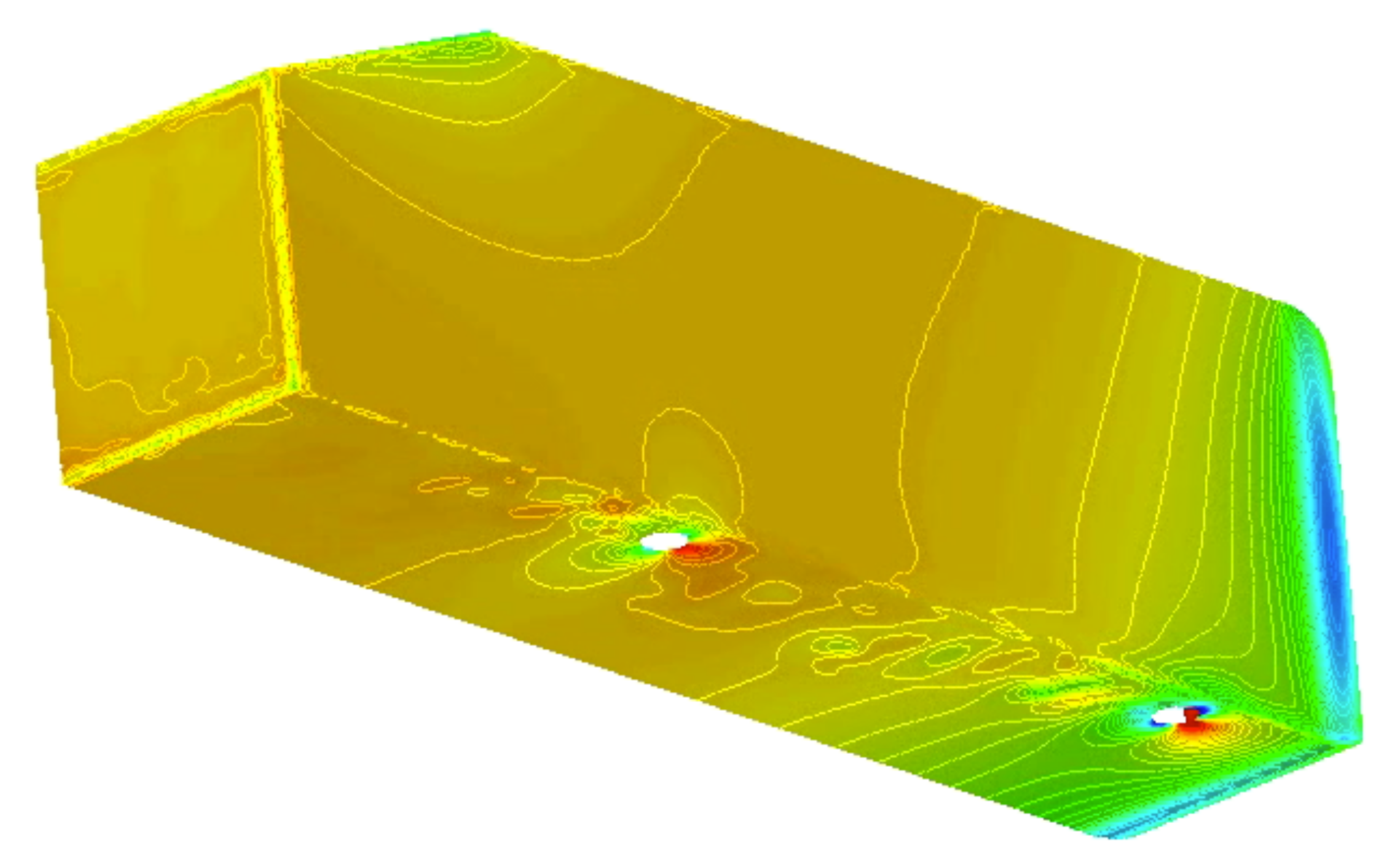}
\label{fig:fomPressure}
}
\subfigure[GNAT ROM equipped with snapshot-collection procedure 1
($\nstate = 283$, $\nR=\nJ = 1514$, and sample mesh with 378 sample nodes)]{
\includegraphics[width=.45\textwidth]{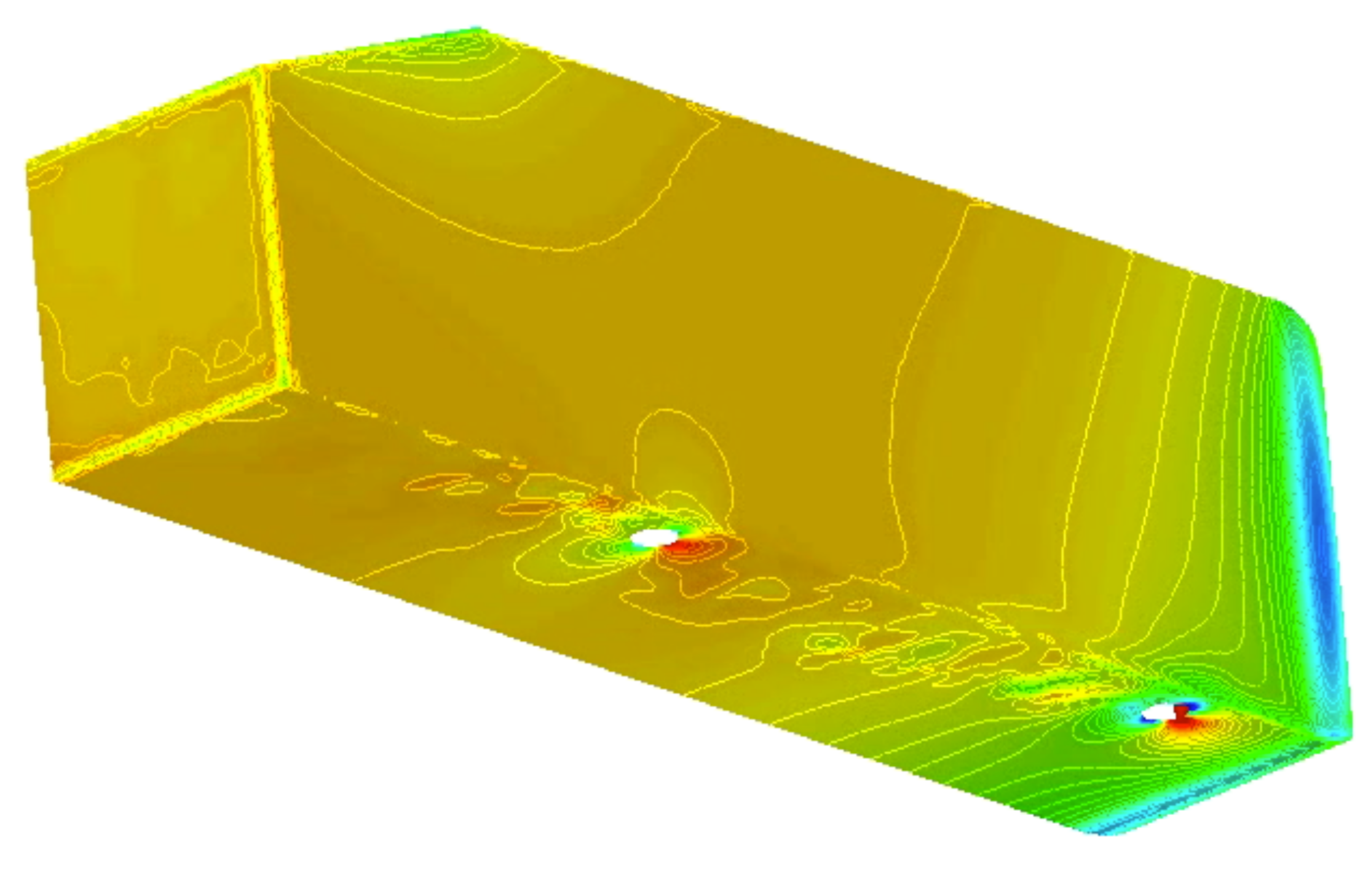}
\label{fig:gapPressure}
}
\caption{Surface-pressure contours at $t = 0.1$ s}
\label{fig:pressureCompare}
\end{figure}

\begin{table}[tb] 
 \caption{Online performance results of GNAT on 4 cores for a sample mesh with
 378 sample nodes}
  \label{tab:snap0snap1} 
 \centering 
 \begin{tabular}{||c||c|c|c|c|c||} 
  \hline 
Snapshot-collection   & \multirow{2}{*}{$\mathrm{RD}$} &
Average \# of Newton    & \multirow{2}{*}{Wall time
(hours)}&\multirow{2}{*}{$\mathrm{\RS}$}&\multirow{2}{*}{$\mathrm{\WT}$}\\ procedure
&               &iterations per time step &          & &
\\ \hline 
0 & 7.43\% & 6.47 & 7.37 &  231 & 1.80  \\
1 & 0.68\% & 2.75 & 3.88 &  438 & 3.42 \\
\hline 
\end{tabular} 
\end{table}

\subsubsection{Effect of node sampling and interpolation vs.\ least-squares approximation}\label{sec:nodeStudy}

To illustrate the effect of the number of sample nodes on GNAT's performance,
this study considers three sample meshes: the sample mesh with 378
sample nodes introduced above (constucted using snapshot-collection procedure
1), a smaller sample mesh with 253 sample nodes,
and a larger one with 505 sample nodes. Algorithm \ref{greedy} is executed to
generate these sample meshes; it employs parameters $\nbasisgreed = 219$ and
$\podres = \podjac$ generated by snapshot-collection procedure
1. Table~\ref{tab:threeSampleMeshes}
reports the characteristics of these sample meshes.  The GNAT models for
these simulations 
are equipped with snapshot-collection procedure 1 and employ
$\nJ=\nR = 1514$ as in the previous section. Because $\neqpernode = 6$, the hyper-reduction
associated with 253 sample nodes corresponds roughly to
interpolation of the residual and its Jacobian. Indeed, the sample-index
factor in this case is $\eta = (253\times 6)/1514 \approx 1$.  For the case of
378 sample nodes, $\eta = 1.5$; the sample mesh with 505 sample nodes is
characterized by $\eta = 2.0$. These latter two cases correspond to
least-squares approximation of the residual and its Jacobian.

\begin{table}[tb] 
\caption{Sample-mesh attributes}
\label{tab:threeSampleMeshes}
\centering 
\begin{tabular}{||c|c|c|c|c||} 
\hline 
\multirow{2}{*}{\# of sample nodes} & \multirow{2}{*}{\# of nodes} &\multirow{2}{*}{\# of elements} & Fraction of nodes    & Fraction of elements\\
                          &                           &                             & of original CFD mesh & of original CFD mesh \\
\hline 
253 & 12808 & 41014 & 0.44\% & 0.24\%  \\
378 & 17096 & 56280 & 0.59\% & 0.33\%  \\
505 & 19822 & 67082 & 0.69\% & 0.39\%  \\
\hline 
\end{tabular} 
\end{table} 

Figure~\ref{fig:sample} reports the time histories of the drag coefficient
obtained using the high-dimensional model and the GNAT ROMs based on these
three sample meshes.  Table~\ref{tab:performanceSample} provides the
performance results for the ROM simulations obtained using 4 cores.
These results indicate the following:
\begin{itemize} 
\item In all cases, GNAT reproduces the time history of the drag coefficient computed using the high-dimensional model with less than 1\% discrepancy.
\item As sample nodes are added, the convergence of Newton's method at each time step improves on average.
\item The fastest performance of GNAT is obtained for the smallest sample mesh.
\item Interpolation of the residual and its Jacobian (253 sample nodes) does
not lead to the best convergence of the Newton solver or the most accurate
results. However, it does lead to the best overall CPU performance of GNAT in
this case.
\end{itemize}

\begin{figure}[htbp] \centering
\includegraphics[width=0.8\textwidth]{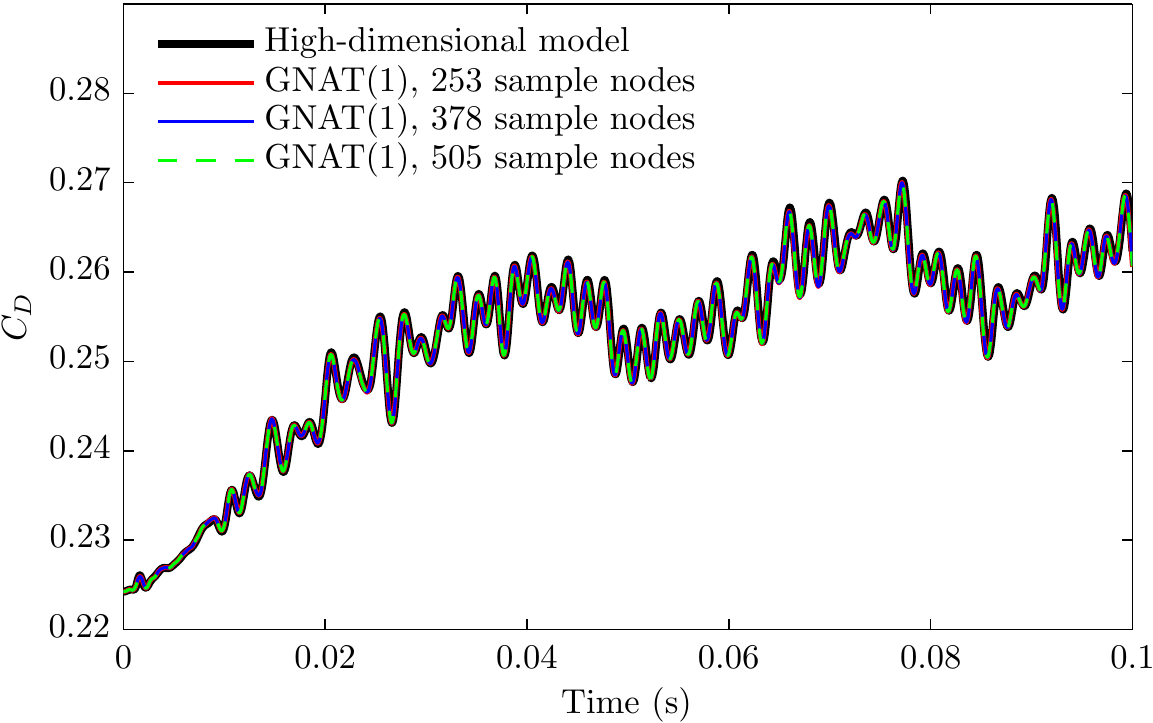} 
\caption{Computed time histories of the drag coefficient for different numbers
of sample nodes (GNAT(1) refers
to GNAT equipped with snapshot-collection procedure 1) }
\label{fig:sample} \end{figure}

\begin{table}[htb] 
\caption{Online performance on 4 cores of GNAT equipped with snapshot-collection procedure 1 for various sample meshes}
\label{tab:performanceSample} 
\centering 
\begin{tabular}{||c|c||c|c|c|c|c||} 
\hline 
\multirow{2}{*}{\# of sample
nodes}&\multirow{2}{*}{$\eta$}&\multirow{2}{*}{$\mathrm{RD}$}&Average \# of
Newton     &\multirow{2}{*}{Wall time (hours)}&\multirow{2}{*}{$\mathrm{\RS}$}
& \multirow{2}{*}{$\mathrm{\WT}$}\\
                                   &                       &
																	 &iterations per time step &
																	 &             &\\
\hline 
253 & $\approx 1$ & 0.79\% & 4.38 &3.77 &  452&  3.52\\
378 & 1.5         & 0.68\% & 2.75 &3.88 &  438&  3.42\\
505 & 2.0         & 0.75\% & 2.25 &4.22 &  403&  3.15\\
\hline 
\end{tabular} 
\end{table} 

\subsubsection{Parallel scalability}\label{sec:gnatParallel}

Due to the sample mesh concept, GNAT is parallelized in the same manner as a
typical CFD code is, using mesh partitioning.  However, because GNAT operates
on a dramatically smaller mesh, its parallel performance cannot be expected to
scale in the strong sense --- that is, for a fixed ROM size and an increasing
number of processors. This is also true for the online stage of any other
model-reduction method.

To obtain an idea of the strong scaling that can be expected from a nonlinear
model-reduction method, Table~\ref{tab:timing} reports the CPU performance
results obtained for GNAT equipped with snapshot-collection procedure 1, the
sample mesh with 378 sample nodes, and $\nJ=\nR = 1514$. Excellent speedups
are obtained for a number of cores varying between 2 and 8, a good speedup is
obtained for 12 cores, and a reasonable one is obtained for 16 cores. For a
larger number of cores, the parallel efficiency (defined as the ratio of the
speedup to the number of cores) increasingly deteriorates.  This is not
surprising given that the GNAT ROM operates on a mesh with only
378 sample nodes.

\begin{table}[tb] 
\caption{Assessment of GNAT's strong scaling performance for a sample mesh
with 378 sample nodes}
\label{tab:timing} 
\centering 
\begin{tabular}{||c|c|c||c|c||} 
\hline 
\# of cores  & Wall time (hours)& Speedup & $\WT$ & $\RS$\\
\hline 
1  & 16.1 & 1.0 &0.83 &422\\
2  & 8.74 & 1.84 & 1.52&388\\
4  & 3.88 & 4.14$^\star$ & 3.43&438\\
8  & 2.50 & 6.44 &5.32 &340\\
12 & 1.94 & 8.25 & 6.86&292\\
16 & 2.08 & 7.74 & 6.39&204\\
\hline 
\end{tabular} 
\\ \small{$^\star$This superlinear speedup is likely due to caching and other
memory management effects.}
\end{table}

\subsubsection{Performance comparison with other function-sampling ROM methods}
\label{sec:fineCompare}
To conclude this section, the performance of GNAT equipped with snapshot-collection procedure 1 and $\nJ=\nR = 1514$
is compared to that of other hyper-reduction techniques based on function
sampling. This study employs the same wake flow problem, the same state POD
basis of dimension $\nstate = 283$, and same sample mesh with 378 sample
nodes. The following function-sampling techniques are compared with GNAT:
\begin{enumerate} 
\item A collocation of the nonlinear equations followed by a Galerkin
projection of the resulting over-determined system of 2268 nonlinear
equations (378 sample nodes $\times$ 6 equations per node) with $\nstate=283$
unknowns \citep{astrid2007mpe,ryckelynck2005phm}.
\item A collocation followed by a least-squares solution of the resulting over-determined
system \citep{LeGresleyThesis}.
\item A discrete empirical interpolation method (DEIM)-like
~\citep{chaturantabut2010journal} approach that employs snapshot-collection
procedure 0 and $\nR=\nJ=2268$ so that the
residual and Jacobian functions are approximated by interpolation.
The tested approach employs the
 tier II Petrov--Galerkin solution of the overdetermined equations as opposed to the
 Galerkin projection; this is done to isolate the effect of the
 hyper-reduction technique on performance.
\end{enumerate}

Figure~\ref{fig:allMethods} reports the time histories of the drag-coefficient
computed using all hyper-reduction techniques outlined above.  Both
collocation approaches lead to nonlinear instabilities after a few time steps
of the flow simulation, thereby exposing the weakness of collocation for
highly nonlinear problems.  The DEIM-like approach, which  employs the popular
but inconsistent snapshot-collection procedure 0, also performs poorly. For
this approach, the Newton iterations begin to generate zero search directions
after only a few time steps of the flow simulation. 

\begin{figure}[htbp] \centering
\includegraphics[width=0.8\textwidth]{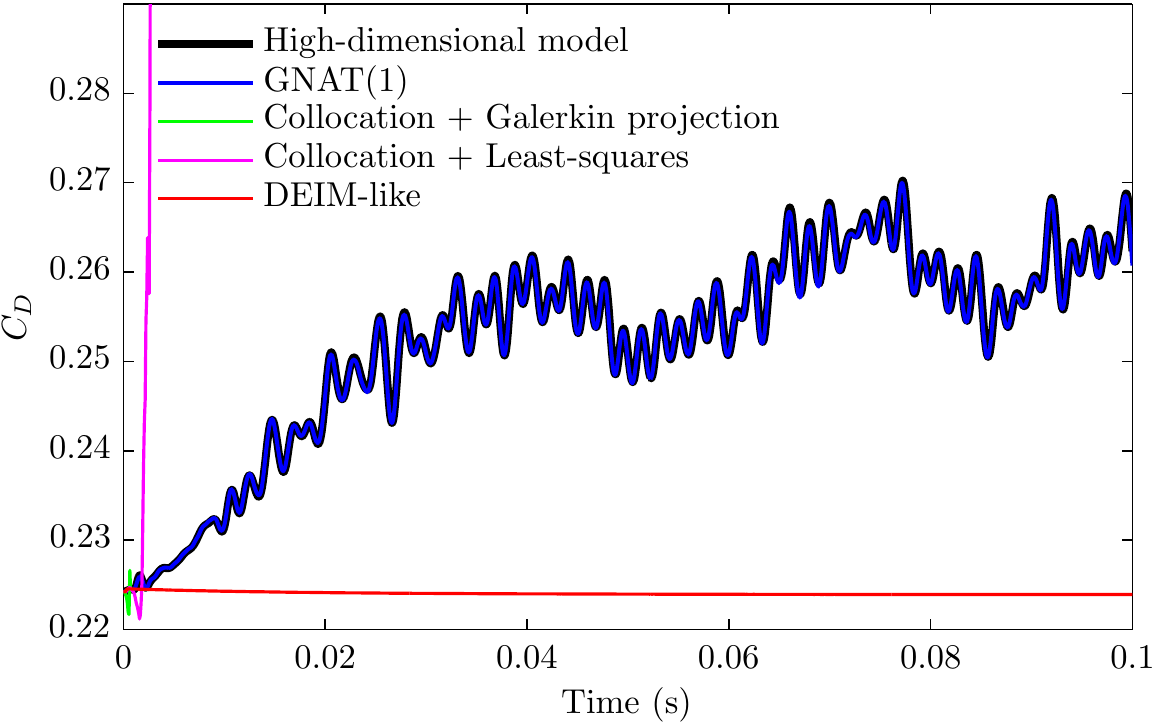} 
\caption{Computed time histories of the drag coefficient
(GNAT(1) refers
to GNAT equipped with snapshot-collection procedure 1)}
\label{fig:allMethods} \end{figure}

\section{Conclusions}
\label{sec:Conc}

In this work, the Gauss--Newton with approximated tensors (GNAT) nonlinear
model-reduction method is equipped with a sample-mesh concept that eases the
implementation of its online stage on parallel computing platforms.  This work
also develops global state-space error bounds that justify GNAT's design,
characterize its mathematical properties, and highlight its advantages in
terms of minimizing components of these bounds.  The effectiveness of GNAT on
parametric problems and its robustness for highly nonlinear computational
fluid dynamics (CFD) applications characterized by moving shocks is
demonstrated by the solution of a conservation problem described by the
inviscid Burgers' equation with a variable source term and boundary condition.
GNAT's ability to reduce by orders of magnitude the core-hours required to
compute turbulent viscous flows at high Reynolds numbers, while preserving
accuracy, is demonstrated with the simulation of the flow field in the wake of
the Ahmed body.  For this popular benchmark problem with over 17 million
unknowns, GNAT is found to outperform several other nonlinear model reduction
methods, reduce the required computational resources by more than two orders
of magnitude, and deliver a solution with less than 1\% discrepancy compared
to its high-dimensional counterpart.

\appendix

\section{Proof of consistent snapshots for the state POD basis}\label{app:snapshotConsistency}
The following proposition proves that two options for collecting state
snapshots lead to a consistent projection. For the sake of simplicity, one set
of training inputs $\paramtrain$ is considered and therefore $\mathcal
D_{\mathrm{train}} = \{\paramtrain\}$. Recall the reduced-order-model
solution is defined by Eq.~\eqref{eq:yApprox} as
\begin{equation}
\label{eq:GNupdate}\tilde\state^{n}(\params)=\initialCond(\params)
+\podstate\deltaStatecoords^{n}(\params),\quad n=0,\ldots,\nt.
\end{equation}

\begin{propos} \label{propos:projconsistent}
\emph{Consistency of the state snapshots.} Assume the following:
 \begin{enumerate} [1.]
	\item\label{ass:snap} The set of snapshots $\snapshotSet{1}$ or $\snapshotSet{2}$ is used to compute $\podstate$ via
	POD, where 
	\begin{align} \label{eq:newConsistentSnapshots} 
	\snapshotSet{1} &\equiv\{\state^n(\paramtrain)-\state^{0}(\paramtrain)\ |\ 
	n=1,\ldots,\nt\}\\
	\snapshotSet{2} &\equiv\{\state^n(\paramtrain)-\state^{n-1}(\paramtrain)\ |\
	n=1,\ldots,\nt\}.
  \end{align} 
\item\label{ass:GN}
The Gauss--Newton method is employed to compute solutions
$\deltaStatecoords^{n+1}(\params)$ to the nonlinear
least-squares problem
\begin{align}\label{eq:GNPGapp}
\deltaStatecoords^{n+1}(\params) &= \arg\min_{ y\in
\RR{\nstate}} 
f^n(y;\params),
\end{align}
for $n=0,\ldots,\nt-1$ and any $\params\in\mathcal D$. Here,
\begin{equation} 
f^n(y;\params)\equiv\frac{1}{2}\|\tilde R^n(\initialCond(\params)+\podstate
y;\params)\|_2^2.
\end{equation} 
The residual $\tilde R^n$ arising from the sequence of reduced-order-model solutions is related to the
residual $R^n$ arising from the sequence of high-dimensional-model solutions as
follows:
\begin{align} 
\tilde R^n(\state;\params) &\equiv S^n(\state,\tilde \state^n,\ldots,\tilde
\state^1,\state^0;\params)\\
S^n(\state,\state^n,\ldots,\state^1,\state^0;\params)&\equiv R^n(\state;\params),
\end{align} 
for $n=0,\ldots,\nt-1$ and any $\params\in\mathcal D$. Here, $S^n$ explicitly
reflects the dependence of the residual on the state at previous time steps.
\item \label{ass:ic}The reduced-order model employs the same initial condition
as the high-dimensional model:
\begin{equation}
\deltaStatecoords^{0}(\params)=0.
\end{equation}
\item \label{ass:gaussNewton} The standard assumptions (see Theorem 10.1
\citep{NocedalWright}) needed for the convergence of the Gauss--Newton iterations to a stationary point contained in the level set $\lnparams$:
 \begin{equation} 
\deltaStatecoords^{n+1}(\params)\in\stationparams, 
\end{equation}
for $n=0,\ldots,\nt-1$  and any $\params\in\mathcal D$. Here, define
\begin{gather} 
\lnparams \equiv \{y\ |\ f^n(y;\params)\leq
f^n\left(\deltaStatecoords^{n+1(0)};\params\right)\}\\
\stationparams \equiv \{y\ |\ y\in\lnparams,\ \nabla
f^n(y;\params)=0\}.
\end{gather} 
\item \label{ass:oneStationary}The level set $\lnparams$ contains only one stationary
point: $|\stationparams| = 1$
for $n=0,\ldots,\nt-1$  and any $\params\in\mathcal D$.
\end{enumerate}
Then, the projection approximation is consistent in the sense that the
Petrov--Galerkin ROM (i.e., GNAT without hyper-reduction) associated with a
POD basis $\podstate$ that is not truncated computes the same states as the
original high-dimensional CFD model for the training inputs --- that is,
\begin{equation} \label{eq:consistentSnapsResult}
\tilde \state^{n}(\paramtrain)=
\state^{n}(\paramtrain), \quad n=0,\ldots,\nt.
\end{equation} 
\end{propos}
\begin{proof}
Consider computing solutions $\tilde\state^n(\paramtrain)$, $n=0,\ldots,\nt$
under the stated assumptions. In the sequel, the argument $\paramtrain$ is
dropped for notational simplicity. The result, i.e., Eq.\
\eqref{eq:consistentSnapsResult}, is proven by induction. It is true for $n=0$
due to Assumption \ref{ass:ic}.
Assume now that $\tilde \state^i = \state^i$, $i=0,\ldots,n$. 

Assumption \ref{ass:GN} ensures that the
Gauss--Newton method is used to compute the solution
$\deltaStatecoords^{n+1}$. Assumption \ref{ass:gaussNewton} guarantees
that these Gauss--Newton iterations will converge to a local stationary point in
the level set $\levelsetTs{n}$. Therefore,
 \begin{equation} \label{eq:converge1}
	\deltaStatecoords^{n+1}\in\stationaryTs{n}.
  \end{equation} 

Assumption \ref{ass:snap} ensures that
\begin{equation} \label{eq:discrepInSpan0}
\state^{n+1}-\state^{0}\in\Range{\podstate}.
\end{equation} 
To see this, first consider the case where $\snapshotSet{1}$ is used to compute
$\podstate$. Then, $\state^{n+1}(\paramtrain)-\state^{0}(\paramtrain)\in\snapshotSet{1}$ 
and therefore
$\state^{n+1}(\paramtrain)-\state^{0}(\paramtrain)\in\Span{\snapshotSet{1}}$.
If the POD basis $\podstate$ is not truncated, then
$\Range{\podstate}=\Span{\snapshotSet{1}}$ and Eq.\
\eqref{eq:discrepInSpan0} holds.
Now, consider the case where $\snapshotSet{2}$ is used for computing $\podstate$. 
Because
$\state^{i+1}(\paramtrain)-\state^{i}(\paramtrain)\in\snapshotSet{2}$,
$i=0,\ldots,\nt$, then
$\state^{n+1}(\paramtrain)-\state^{0}(\paramtrain)\in\Span{\snapshotSet{2}}$. 
If the POD basis $\podstate$ is not truncated, then
 $\Range{\podstate}=\Span{\snapshotSet{2}}$, and again Eq.\ \eqref{eq:discrepInSpan0} holds.  

The induction assumption ($\tilde \state^i = \state^i, i=0,\ldots,n$) and
Eq.~\eqref{eq:discrepInSpan0} ensure that
\begin{equation}\label{eq:fomSolInStationary0}
\podstate^T\left(\state^{n+1}-\state^{0}\right)\in
\stationaryTs{n}.
\end{equation}
This can be derived by setting  $y =
\podstate^T\left(\state^{n+1}-\state^{0}\right)$ and writing the objective
function:
 \begin{align} 
f^n(\podstate^T\left(\state^{n+1}-\state^{0}\right))&=\frac{1}{2}\|\tilde
R^n(\state^{0}+\podstate
\podstate^T\left(\state^{n+1}-\state^{0}\right))\|_2^2\\
\label{eq:resZeroTilde0}&=\frac{1}{2}\|\tilde
R^n(\state^{n+1})\|_2^2\\
&=\frac{1}{2}\|
\label{eq:resZero0}R^n(\state^{n+1})\|_2^2\\
\label{eq:resZero0Final}&=0.
  \end{align} 
Eq.\ \eqref{eq:resZeroTilde0} is due to Eq.\ \eqref{eq:discrepInSpan0} and the
orthogonality of the POD basis. Eq.\ \eqref{eq:resZero0} arises from
the equalities
\begin{equation}
\tilde R^n(w) = S^n(w,\tilde \state^n,\ldots,\tilde \state^1,\initialCond) =
S^n(\state, \state^n,\ldots, \state^1,\initialCond) = R^n(\state),
\end{equation}
which hold due to the induction assumption.
Finally, Eq.\
\eqref{eq:resZero0Final} holds because the full-order solution
leads to a zero residual: $R^n(\state^{n+1})=0$. Because $f^n(y) \geq
0$ $\forall y$, Eq.\ \eqref{eq:resZero0Final} implies that
$\podstate^T\left(\state^{n+1}-\initialCond\right)$ is a local minimizer of $f^n$,
so Eq.\ \eqref{eq:fomSolInStationary0} holds.

Assumption \ref{ass:oneStationary}, Eq.\ \eqref{eq:converge1}, and Eq.\
\eqref{eq:fomSolInStationary0} together imply
\begin{equation} \label{eq:deltaSame0}
\deltaStatecoords^{n+1} = \podstate^T\left(\state^{n+1}-\initialCond\right).
\end{equation} 
Substituting Eq.\ \eqref{eq:deltaSame0} into Eq.\ \eqref{eq:GNupdate} 
yields
\begin{align} \label{eq:almostDone}
\tilde\state^{n+1} = \initialCond + \podstate
\podstate^T\left(\state^{n+1}-\initialCond\right).
\end{align} 
Eq.\ \eqref{eq:almostDone} along with Eq.\ \eqref{eq:discrepInSpan0} and the
orthogonality of the POD basis provides the result:
\begin{equation} 
\boxed{
\tilde\state^{n+1}(\paramtrain) = \state^{n+1}(\paramtrain),\quad
n=0,\ldots,\nt.}
\end{equation} 
\end{proof}
\section{Further discussion of the various snapshot-collection procedures}
\label{sec:SCP}

If $\podjac = \podres$, then $A=B$ and $B\restrict{\jk\podstate} = \displaystyle{\frac{\partial [B\restrict{R(\state^{(0)} + \podstate y)]}}{\partial y}}$. As a result, 
the GNAT iterations (\ref{eq:newtonRed1approx})--(\ref{eq:newtonRed2approx}) are in this case equivalent to the Gauss--Newton iterations for solving
\begin{equation} 
\underset{\bar \state\in
\state^{(0)}+\mathcal W}{\mathrm{minimize}}\  \|B\restrict{R(\bar
\state)}\|_2.
  \end{equation} 
Because $\|B\restrict{R}\|_2 = \|\podres B\restrict{R}\|_2$ when $\podres^T\podres = I$ and the gappy POD approximation of $R$ is $\tilde R = \podres B\restrict{R}$, 
the GNAT iterations are also equivalent to the Gauss--Newton iterations for solving
 \begin{equation} 
\underset{\bar \state\in
\state^{(0)}+\mathcal W}{\mathrm{minimize}}\  \|\tilde R(\bar \state)\|_2.
  \end{equation} 
Therefore, when $\podjac=\podres$ and $\podres$ has orthonormal columns, GNAT
inherits the convergence properties of the Gauss--Newton method. This is the
rationale behind both procedure 0 and procedure 1 outlined in Section \ref{sec:gapConsistency}.

On the other
hand, procedure 2 and procedure 3 use different bases $\podres$ and $\podjac$.
For this reason, the GNAT iterations
(\ref{eq:newtonRed1approx})--(\ref{eq:newtonRed2approx}) cannot be associated
with Gauss--Newton iterations for nonlinear residual minimization.
Furthermore, choosing $\podjac\neq \podres$ causes the least-squares problem
(\ref{eq:newtonRed1approx}) to try to `match' quantities that lie in different
subspaces. For these reasons, procedure 2 and procedure 3 may lack robustness
and experience convergence difficulties as reported in Ref.~\citep{carlbergThesis}. 
\section{Error bounds for the solution computed by a discrete nonlinear model
reduction method}
\label{app:EB}
This section proves the error bound~(\ref{eq:boundBackEuler}) presented in
Section~\ref{sec:errBackEuler}.  For the sake of notational simplicity, the
derivation presented here considers the approximation error arising from a
given set of inputs and therefore omits $\params$ from the arguments of the
nonlinear functions. Rewriting the residual~\eqref{eq:backEulerResidual}
in this fashion leads to
\begin{gather}\label{eq:FOMres}
R^n(\statenp) = \statenp-\staten-\Delta t F(\statenp,\tnp).
\end{gather}
Similarly, the residual at the the $n$-th time step arising from any sequence
of approximate solutions $\stateApproxn$, $n=0,\ldots,\nt$, e.g., generated by
a discrete nonlinear ROM, for the same input parameters can be written as
\begin{gather}\label{eq:ROMres}
\Rapprox^n(\stateApproxnp) = \stateApproxnp-
\stateApproxn-\Delta t F(\stateApproxnp,\tnp).
\end{gather}

Subtracting \eqref{eq:ROMres} from \eqref{eq:FOMres} yields
\begin{equation} 
R^n(\statenp) -  \Rapprox^n(\stateApproxnp)= \statenp-\staten-\Delta
t F(\statenp,\tnp) -  \stateApproxnp+ \stateApproxn+\Delta t
F(\stateApproxnp,\tnp).
 \end{equation} 
The above expression can be re-arranged as
 \begin{equation}\label{eq:resAlmostDone2} 
\statenp -  \stateApproxnp-\Delta t F(\statenp,\tnp) +\Delta t
F(\stateApproxnp,\tnp)=R^n(\statenp) -  \Rapprox^n(\stateApproxnp) +\staten-
\stateApproxn.
 \end{equation} 
Introducing $f:(x,t)\mapsto  x - \Delta t F(x,t)$ and the inverse Lipschitz
constant\footnote{Note that $\varepsilon = \frac{1}{\Lgn}$ in
Eq.~\eqref{eq:firstLipschitz}.}
 \begin{equation} 
 \Lgn \equiv \sup_{x\neq y}\frac{\|x-y\|}{\|f(x,\tnp)-f(y,\tnp)\|}
 \end{equation} 
allows Eq.~(\ref{eq:resAlmostDone2}) to be transformed into the following
bound on the \emph{local} approximation error:
 \begin{equation} \label{eq:lgnBound}
\|\statenp-\stateApproxnp\|\leq\Lgn \left(\epsilon_\mathrm{Newton} + \|
\Rapprox^n(\stateApproxnp)\|  +\|\staten- \stateApproxn\|\right).
 \end{equation} 

Assuming
that the initial approximation error is zero\footnote{This is valid for both the Petrov--Galerkin and GNAT ROMs as they employ the same initial
condition as the high-dimensional model (See Algorithm \ref{alg:onlinegnat}).} ($\stateApprox^0 = \state^0$), the inequality~(\ref{eq:lgnBound}) leads to the following result
 \begin{equation} 
\|\staten-\stateApproxn\| \leq \sum_{k=1}^{n}a^{k}b_{n-k},
 \end{equation} 
where $a = \Lg\equiv \sup_{n\in\{1,\ldots,\nt\}}\Lgn$ and 
\begin{equation}\label{eq:andef}
b_n\equiv \epsilon_\mathrm{Newton} + \| \Rapprox^n(\stateApproxnp)\|.
\end{equation}

From the triangle inequality, it follows that $\| \Rapprox^n(\stateApproxnp)\|\leq \|\proj\Rapprox^n(\stateApproxnp)\| + \|\left(I-\proj\right)\Rapprox^n(\stateApproxnp)\|$
for any $\proj$. Hence, another bound for the approximation error is
 \begin{equation} \label{eq:bound2}
\|\staten-\stateApproxn\| \leq \sum_{k=1}^{n}a^{k}c_{n-k},
 \end{equation} 
where
\begin{equation}\label{eq:Cndef}
c_n\equiv \epsilon_\mathrm{Newton} + \| \proj\Rapprox^n(\stateApproxnp)\| +
\|(I-\proj)\Rapprox^n(\stateApproxnp)\|
\end{equation}
and $c_n\geq b_n$.
The bound \eqref{eq:bound2} is particularly interesting for the case where
$\proj = \podres\hatpodrespseudo\restrict{}$ represents the gappy POD
operator because
$\|\proj\Rapprox^n(\stateApproxnp)\|=\|\hatpodrespseudo\restrict{\Rapprox^n(\stateApproxnp)}\|$
is readily computable by GNAT.

In \ref{app:gappyPODbound}, it is shown that an upper bound for the gappy POD approximation error is
\begin{equation} \label{eq:gappyPODbound}
\|\left(I - \proj\right) \Rapprox^n(\stateApproxnp)\|\leq \|\mathsf R^{-1}\|\|(I - \mathbb P) \Rapprox^n(\stateApproxnp)\|, 
 \end{equation} 
where $\mathbb P=\podres\podres^T$ defines the orthogonal projector onto
$\Range{\podres}$, and $\restrict{\podres} = \mathsf Q\mathsf R$ is
the thin QR factorization of $\restrict{\podres}$ with $\mathsf Q\in
\RR{\nin\times \nR}$ and $\mathsf R \in \RR{\nR\times\nR}$. Therefore from \eqref{eq:gappyPODbound}, it follows that yet another error bound for the approximation error is
 \begin{equation}\label{eq:totalBound1} 
\|\staten-\stateApproxn\| \leq \sum_{k=1}^{n}a^{k} d_{n-k},
 \end{equation} 
where
 \begin{equation} 
 d_n \equiv \epsilon_\mathrm{Newton} + \|\proj \Rapprox^n(\stateApproxnp)\| +
 \|\mathsf R^{-1}\|\|\left(I - \mathbb P\right) \Rapprox^n(\stateApproxnp)\|.
 \end{equation} 
Because $b_n \leq c_n\leq d_n$, it follows that a global bound for the
approximation error at the $n$-th time step with $1\leq n\leq \nt$ is given by
 \begin{equation} 
\boxed{\|\staten-\stateApproxn\| \leq
\sum_{k=1}^{n}a^{k}b_{n-k}\leq\sum_{k=1}^{n}a^{k}c_{n-k}\leq\sum_{k=1}^{n}a^{k}d_{n-k}.}
 \end{equation}

\section{Error bound for the gappy POD approximation}\label{app:gappyPODbound}
This section establishes a bound for the error associated with the gappy POD approximation
of a vector $g\in\RR{N}$ using a POD basis $\Phi_f\in\RR{N\times n_f}$ and a
set of $\nin\geq n_f$ sample indices $\mathcal I$ that define the sample matrix $Z$
(see Section \ref{sec:sysApproxOptimality} for these definitions).\footnote{This
development follows closely the proof of Lemma 3.2 in
Ref.~\citep{chaturantabut2010journal}.} 

Define 
$g^*\equiv \mathbb Pg$ with $\mathbb P\equiv \Phi_f\Phi_f^T$ as the orthogonal
(i.e., optimal) projection of 
$g$ onto
$\Range{\Phi_f}$. Also, define the difference between $g$ and its
orthogonal projection onto $\Range{\Phi_f}$ as $e\equiv g-g^*$.
Finally, define the gappy POD projection matrix $\proj \equiv \Phi_g\mathsf R^{-1}\mathsf Q^T\restrict{}$, where
$\restrict{\Phi_f} = \mathsf Q\mathsf R$ is
the thin QR factorization of $\restrict{\Phi_f}$ with $\mathsf Q\in
\RR{\nin\times n_f}$ and $\mathsf R \in \RR{n_f\times n_f}$.

The gappy POD approximation of $g$ is $\proj  g = \proj \left(e+g^*\right)$. It can also be written as
 \begin{equation} \label{eq:gappyErrorEquation}
	\proj g= \proj e + g^* 
  \end{equation} 
because $\proj g^* = g^*$, as $g^*\in\Range{\Phi_f}$. Substituting $g^* =
g-e$ into Eq.\ \eqref{eq:gappyErrorEquation} yields
$(I-\proj)g = (I-\proj )e$. Therefore,
\begin{equation}\label{eq:gBound1}
\|(I-\proj)g \|_2 = \|(I-\proj )e\|_2\leq \|(I-\proj )\|_2\|e\|_2 .
\end{equation}
Because $\|I-\proj \|_2 = \|\proj \|_2$ for any projection matrix $\proj
$ not equal to 0 or $I$, it follows that
\begin{equation}\label{eq:IP}
\|I-\proj \|_2 = \|\proj \|_2 =
\|\Phi_g\mathsf R ^{-1}\mathsf Q^T\restrict\|_2=\|\mathsf R^{-1}\|_2.
\end{equation}
The last equality follows from the fact that $\Phi_g$, $\restrict{}^T$, and
$\mathsf Q$ have orthonormal columns. Substituting \eqref{eq:IP} in \eqref{eq:gBound1} gives the result
\begin{equation}
\boxed{\|(I-\proj) g\|_2\leq
\|\mathsf R^{-1}\|_2\Big\|\left(I-\mathbb P\right)g\Big\|_2.}
\end{equation}
\section*{Acknowledgments}

Most of this work was completed while the first and third authors were at
Stanford University. The authors thank Phil Avery and Charbel Bou-Mosleh for
their contributions to the parallel implementation of GNAT in AERO-F, and Matthew Zahr
for his contribution to the Burgers' equation example.  All authors acknowledge
partial support by the Motor Sports Division of the Toyota Motor Corporation
under Agreement Number 48737, and partial support by the Army Research
Laboratory through the Army High Performance Computing Research Center under
Cooperative Agreement W911NF-07-2-0027.  The first author also acknowledges
partial support by the National Science Foundation Graduate Fellowship, the
National Defense Science and Engineering Graduate Fellowship, and an
appointment to the Sandia National Laboratories Truman Fellowship in
National Security Science and Engineering. The Truman Fellowship is
sponsored by Sandia Corporation (a wholly owned subsidiary of Lockheed
Martin Corporation) as Operator of Sandia National Laboratories under its
U.S.\  Department of Energy Contract No.\ DE-AC04-94AL85000.  The content of
this publication does not necessarily reflect the position or policy of any
of these institutions, and no official endorsement should be inferred. 
\bibliographystyle{model1-num-names}
\bibliography{references}
\end{document}